\documentclass[pdflatex,sn-mathphys-num]{sn-jnl}

\usepackage{graphicx}%
\usepackage{multirow}%
\usepackage{amsmath,amssymb,amsfonts}%
\usepackage{amsthm}%
\usepackage{mathrsfs}%
\usepackage[title]{appendix}%
\usepackage{xcolor}%
\usepackage{textcomp}%
\usepackage{manyfoot}%
\usepackage{booktabs}%
\usepackage{algorithm}%
\usepackage{algorithmicx}%
\usepackage{algpseudocode}%
\usepackage{listings}%

\newtheorem{theorem}{Theorem}
\newtheorem{proposition}{Proposition}%
\newtheorem{example}{Example}%
\newtheorem{remark}{Remark}%

\newtheorem{definition}{Definition}%
\newtheorem{lemma}{Lemma}%

\begin{document}

\title[Article Title]{On Optimality Conditions for Mathematical Programming Problems Based on Strong Subdifferentials}

\author[1]{\fnm{Felipe} \sur{Lara}}\email{flarao@academicos.uta.cl}

\author*[2]{\fnm{Alberto} \sur{Ramos}}\email{aramosf@academicos.uta.cl}

\equalcont{These authors contributed equally to this work.}

\affil[1]{\orgdiv{Instituto de Alta Investigación (IA)}, \orgname{Universidad de Tarapacá}, \state{Arica}, \country{Chile}}

\affil*[2]{\orgdiv{Facultad de Ciencias}, \orgname{Universidad de Tarapacá}, \state{Arica}, \country{Chile}}

\abstract{We develop refined Karush-Kuhn-Tucker (KKT) and Fritz-John (FJ)-type optimality conditions for nonsmooth, nonconvex mathematical pro\-gra\-mming problems. We pay special attention in the case that the functional constraint belongs to a specific class of generalized convex functions known as strongly quasiconvex functions. After analyzing a specialized sub\-di\-ffe\-ren\-tial, named the strong subdifferential, we compute the normal cone of the supremum function in terms of such subdifferentials, and apply this result to the mathematical programming problem.  
We illustrate our important results by examples.}




\maketitle

\section{Introduction}

Quasiconvex functions are characterized by having convex sublevel sets, a fundamental property with significant implications in economic theory. This concept is particularly relevant in consumer preference theory, notably due to Gerard Debreu's seminal work \cite{D-1959}, where quasiconvexity mathematically formalizes the natural assumption of a {\it tendency to diversification} in consumer behavior.  

From a mathematical perspective, functions with convex sublevel sets play a crucial role in ensuring existence of minimizers, deriving optimality conditions, and designing iterative algorithms. 
The key insight is that minimizing a function over the entire space is equivalent to minimizing it over any of its sublevel sets, provided that these sets are convex. This property simplifies optimization problems by restricting the search domain while preserving global optimality.   

In the differentiable case, quasiconvex functions exhibit another outstanding property: they are decreasing on their sublevel sets. This characteristic has significant implications, one of the most notable being that KKT optimality conditions become sufficient under quasiconvexity assumptions on both the constraint and objective functions. These foundational results were first established in the seminal work of Arrow and Enthoven \cite{AE} and later refined by various authors, including Mangasarian in \cite{Manga}.  

For decades, quasiconvexity, KKT-type optimality conditions, and economic applications have been closely intertwined, leading to widespread applications across diverse fields such as engineering, finance, management, and computer science, among others. 
However, a major limitation of quasiconvex functions is that they may exhibit flat regions outside their set of minimizers, posing challenges for optimization algorithms. In such cases, iterative methods often stagnate at critical points that are far from the true minimizers.   
To address this issue, several subclasses of quasiconvex functions have been proposed in the literature, but until now, linear convergence of standard first-order methods, such as gradient descent \cite{LMV} and proximal-point-type algorithms \cite{IL-7,ILMY,Lara-9}, has only been guaranteed for strongly quasiconvex functions. This class, introduced by Polyak \cite{P} in the context of minimizing sequences for extremum problems, has gained considerable attention in recent years (see, e.g., \cite{NS,Pi,Tran}).  

In this paper, we develop KKT and FJ optimality conditions for nonsmooth nonconvex mathematical programming problems under the assumption that the constraint functions are strongly quasiconvex. 
Our KKT and FJ optimality conditions are described in terms of the strong subdifferential, a specialized tool introduced to analyzing nonsmooth strongly quasiconvex functions see \cite{Kab-Lara-2, Lara-9},  
since it can be nonempty in situations where classical subdifferentials (as limiting, or regular) are empty sets, and thus provides some useful information. 
 


The structure of the paper is as follows. In Section \ref{sec:2}, we recall basic notions from variational analysis and generalized convexity. In Section \ref{sec:03}, we present new properties and calculus rules for the strong subdifferential and its relationship with other well-known subdifferentials for strongly quasiconvex functions. 
In Section \ref{sec:04}, we establish a friendly description of the normal cone of the 0-level set of the supremum function, in terms of strong subdifferentials. 
In Section \ref{sec:FJKKT}, we give our KKT and FJ type optimality conditions. Several examples are given to compare our results with the existing literature. 
Finally, conclusions and future research directions are described in Section \ref{sec:conclusion}.

\section{Preliminaries}\label{sec:2}

\subsection{Basic tools on variational analysis}\label{subsec:variationalanalysis}

The inner product in $\mathbb{R}^{n}$ and the Euclidean norm are denoted by $\langle \cdot,\cdot \rangle$ and $\lVert \cdot \rVert$, respectively. The set $]0, + \infty[$ is denoted by $\mathbb{R}_{++}$. 
For $a\in \mathbb{R}^{n}$, we denote $a^+:=\max\{a,0\}$ its non-negative part, where the maximum is taken componentwise. 
Set $I$ a finite index set, the sum $\sum_{j \in I}b_{j}=0$ if $I=\emptyset$.
Given $A$ and $B$ subsets of $\mathbb{R}^{n}$, $A+B$ is the Minkowski sum. 
We set $A+\emptyset=\emptyset$.

For a set $C \subset \mathbb{R}^{n}$, we use $\text{int}\,C$, $\text{cl}\,C$ and $\text{co}\,C$ to denote its interior, closure and convex hull respectively. 
We denote by $\delta_{C}$ the indicator function, i. e, $\delta_{C}(x)=0$, if $x \in C$ and $\delta_{C}(x)=\infty$ otherwise.
For a cone $K$, its {\it polar} is $K^{\circ}$ is $\{z \in \mathbb{R}^{n} \mid \, \langle z, y \rangle \leq 0, \, \forall ~ y \in K\}$.
For a given subset $C \subset \mathbb{R}^{n}$ not necessarily convex, we set $C^{\infty} := \{v \in \mathbb{R}^{n} \mid \, \exists \, t_{k} \rightarrow 0^{+}, \, \exists \, \{v^{k}\}_{k} \subset C ~ \text{ with} ~ t_{k} v^{k} \rightarrow v\}$.

Given an extended-valued function $h: \mathbb{R}^{n} \rightarrow \overline{\mathbb{R}} := \mathbb{R} \cup \{\pm \infty\}$, the effective domain of $h$ is ${\rm dom}\,h := \{x \in \mathbb{R}^{n} \mid h(x) < + \infty \}$, and $h$ is proper if ${\rm dom}\,h$ is nonempty and $h(x) > - \infty$ for all $x \in \mathbb{R}^{n}$. 
We set by ${\rm argmin}_{\mathbb{R}^{n}} h$ the set of all minimal points of $h$. 


For a given set-valued mapping $\mathcal{K}: \mathbb{R}^s\rightrightarrows\mathbb{R}^n$, the \textit{sequential Painlevé-Kuratowski outer/upper limit} of $\mathcal{K}(z)$ as $z\rightarrow z^*$ is defined as 
\begin{equation*}
\limsup_{z \to z^*} \mathcal{K}(z) = \{y^*\in \mathbb{R}^n \mid \exists \, (z^k, y^k) \to (z^*, y^*)  \text{ with } y^k \in \mathcal{K} (z^k), ~ \forall ~ k \in \mathbb{N} \}, 
\end{equation*}
and the \textit{sequential Painlevé-Kuratowski inner limit} of $\mathcal{K}(z)$ as $z\rightarrow z^*$ by
\begin{equation*}
\liminf_{z \to z^*} \mathcal{K}(z) = \{y^*\in \mathbb{R}^n \mid \forall \, z^k \to z^*, \exists \, y^k \to y^*  \text{ such that } y^k \in \mathcal{K} (z^k), \ \forall ~ k \in \mathbb{N} \}.
\end{equation*}

In the following, we recall some basic concepts in variational analysis. 
Given a lower semicontinuous (lsc henceforth) function $h$, the {\it regular subdifferential} of $h$ at $\bar{x} \in \text{dom }h$ is the set
\begin{equation}
 \widehat{\partial} h(\overline{x}) := \{ v \in \mathbb{R}^{n} \mid h(y) \geq h(\overline{x}) + \langle v, y - \overline{x} \rangle+o(\|y-\overline{x}\|), ~ \forall ~ y \in \mathbb{R}^{n}\},
\end{equation}
with $\partial h(x) = \emptyset$ if $x \not \in \mathrm{dom}\,h$.
The {\it Fenchel-Moreau subdifferential} of $h$ at $\bar{x}$ is given by 
   \begin{align*}
\partial^{FM} h(\overline{x}):=
\{ v \in \mathbb{R}^{n}: h(y) \geq h(\overline{x}) + \langle v, y - \bar{x} \rangle, \ \forall y\in \mathbb{R}^{n}\},     
   \end{align*} 
   with $\partial^{FM} h(x) = \emptyset$ if $x \not \in \mathrm{dom}\,h$.
Clearly, $\partial^{FM}h(x)\subset \widehat{\partial}h(x)$, $\forall x$.  
The {\it limiting subdifferential} of $h$ at $\bar{x}$ is defined as 
\begin{align*}
 \partial h(\bar{x}) = \{ v \in \mathbb{R}^n \mid \, \exists \, x^{k} \rightarrow_{h} \bar{x}, \exists \, v^{k} \rightarrow v \text{ with } v^{k} \in \widehat{\partial} h(x^{k}) \}      
\end{align*}
and the {\it horizon subdifferential } of $h$ at $\bar{x}$ is defined by
\begin{align*}
 \partial^{\infty} h(\bar{x}) = \{ v \in \mathbb{R}^n \mid \, \exists \, x^{k} \rightarrow_{h} \bar{x}, \exists \, v^{k} \rightarrow v, \exists \, t_{k} \rightarrow 0^{+} \text{ with } v^{k}\in \widehat{\partial} h(x^{k}); \ t_{k} v^k \rightarrow v \},       
\end{align*}
where $x^{k} \rightarrow_{h} \bar{x}$ means that $x^{k} \rightarrow \bar{x}$ and $h(x^{k})  \rightarrow h(\bar{x})$.

We point out that a lower lsc function $h$ is locally Lipschitz continuous at $\bar{x}$ if and only if $\partial^{\infty} h(\bar{x})=\{0\}$, see \cite[Theorem 9.13]{rwets}. 
If $h$ is strict differentiable function at $\bar{x}$, we get $\partial h(\bar{x})=\widehat{\partial} h(\bar{x})=\{\nabla h(\bar{x})\}$. 
Furthermore, we say that a lsc function $h$ is {\it subdifferential regular} at $\bar{x}$ when $\partial h(\bar{x}) = \widehat{\partial} h(\bar{x})$.

Given a closed set $\mathcal{X} \subset \mathbb{R}^n$ and 
$z \in \mathcal{X}$, the {\it tangent cone}
to $\mathcal{X}$ at $z$ is 
     \begin{equation*}
 T(\mathcal{X}, z):=\{d \in \mathbb{R}^{s} \mid
\exists \, t_{k} \downarrow 0, \
d_{k}\to d \text{ with } z+t_{k}d_{k} \in \mathcal{X}, k \in \mathbb{N}\}.
     \end{equation*}

The {\it regular normal cone} to $\mathcal{X}$ at $z$ is $ \widehat{N}(\mathcal{X}, z):=\widehat{\partial} \delta_{\mathcal{X}}(z)$ and the {\it limiting normal cone } to $\mathcal{X}$ at $z$ is $N(\mathcal{X}, z):=\partial \delta_{\mathcal{X}}(z)$. Furthermore, by \cite{rwets}, we have $\widehat{N}(\mathcal{X}, z)=T(\mathcal{X}, z)^{\circ}$.
The set of feasible directions is $D(\mathcal{X}, z):=\mathbb{R}_{+}(\mathcal{X}-z)$.
  
\subsection{Generalized Convexity}\label{subsec:strongquasiconvex} 

In this subsection, we recall some basic definitions of generalized convexity, including the notion of quasiconvex functions. 
We pay special attention to strongly quasiconvex functions, since they form an important family of functions where numerical methods such as gradient and proximal point algorithms can be applied properly (see \cite{Lara-9,LMV}).

Consider a function $h$ with convex domain. We say that $h$ is 
\begin{itemize}
 \item[$(a)$] convex if $h(\lambda x + (1-\lambda)y) \leq \lambda h(x) + (1 - \lambda) h(y)$ for all $\lambda \in [0, 1]$ and all $x, y \in \mathrm{dom}\,h$.

 \item[$(b)$] strongly convex on ${\rm dom}\,h$ with modulus $\gamma >0$ if for all $x, y \in \mathrm{dom}\,h$ and all
 $\lambda \in[0, 1]$, we have $ h(\lambda y + (1-\lambda)x) \leq \lambda h(y) + (1-\lambda) h(x)-\lambda (1 - \lambda) \frac{\gamma}{2} \lVert x - y \rVert^{2}$.

 \item[$(c)$] quasiconvex if $h(\lambda x + (1-\lambda) y) \leq \max \{h(x), h(y)\}$ for all $\lambda \in [0, 1]$ and all $x, y \in \mathrm{dom}\,h$.

 \item[$(d)$] strongly quasiconvex \cite{P} on ${\rm dom}\,h$ with modulus
 $\gamma \geq 0$ if for all $x, y \in \mathrm{dom}\,h$
 and all $\lambda \in[0, 1]$, we have
 \begin{equation}\label{strong:quasiconvex}
  h(\lambda y + (1-\lambda)x) \leq \max \{h(y), h(x)\} - \lambda(1 -
  \lambda) \frac{\gamma}{2} \lVert x - y \rVert^{2}. 
 \end{equation}
 \noindent It is said that $h$ is strictly convex (resp. strictly quasiconvex) if the inequality in the definition is strict whenever $x \neq y$.
 \end{itemize}
 The relationship between all these notions is summarizing below:
 \begin{align}\label{scheme}
  \begin{array}{ccccccc}
  {\rm strongly ~ convex} & \Longrightarrow & {\rm strictly ~ convex} &
  \Longrightarrow & {\rm convex} \notag \\
  \Downarrow & \, & \Downarrow & \, & \Downarrow \\
  {\rm strongly ~ qcx}\ (\gamma>0) & \Longrightarrow & {\rm strictly ~ qcx} &
  \Longrightarrow & {\rm qcx} 
  \end{array}, 
 \end{align}

where, we denote quasiconvex by qcx.
All the reverse statements may not hold. 
Indeed, $h_{1} (x) = \sqrt{\lVert x \rVert}$ is strongly quasiconvex on any bounded convex set on $\mathbb{R}^{n}$ (see \cite[Theo\-rem 17]{Lara-9}) without being convex  and $h_{2} (x) = \frac{x}{1 + \lvert x \rvert}$ is strictly quasiconvex without being strongly quasiconvex on $\mathbb{R}$. 
Other counterexamples are well-known (see \cite{CM-Book,HKS}).


\begin{remark}\label{rem:exam}
An important family of strongly quasiconvex is given by the quotient of quadratic functions with applications in fractional mathematical programming problems. 
We have the following statement as consequence of \cite[Proposition 4.1]{ILMY}.

 Let $A, B \in \mathbb{R}^{n\times n}$, $a, b \in \mathbb{R}^{n}$, $\alpha, \beta \in \mathbb{R}$, and $h: \mathbb{R}^{n} \rightarrow \mathbb{R}$ be the function given by:
  \begin{equation}
   h(x) = \frac{f(x)}{g(x)} = \frac{\frac{1}{2} \langle Ax, x \rangle + \langle a, x \rangle + \alpha}{\frac{1}{2} \langle Bx, x \rangle + \langle b, x \rangle + \beta}.
  \end{equation}
  Take $m, M>0$ such that $0 < m < M$. We set $K := \{x \in \mathbb{R}^{n}: ~ m \leq g(x) \leq M\}$. 
   
   Suppose that $A$ is positive definite and at least one of the following conditions holds:
  \begin{enumerate}
   \item[$(a)$] $B = 0$ (the null matrix),
   \item[$(b)$] $f$ is nonnegative on $K$ and $B$ is negative semidefinite,
   \item[$(c)$] $f$ is nonpositive on $K$ and $B$ is positive semidefinite,
 \end{enumerate}
  then $h$ is strongly quasiconvex on $K$ with modulus $\gamma =\lambda_{\min} (A)/M$ where $\lambda_{\min} (A)$ is the minimum eigenvalue of $A$.
\end{remark}

We recall that $\bar{x}$ is a $\gamma$-strong minimum if $h(x)\geq h(\bar{x})+\gamma\|x-\bar{x}\|^{2}$ for all $x \in K$ \cite[Definition 22]{Kab-Lara-2}. 
For a further study on generalized convexity and strongly quasiconvex functions, we refer to \cite{AE,ADSZ,CM-Book,HKS,Kab-Lara-2,Lara-9,LMV,Manga} and references therein.

Following \cite{AusselSVVA2015}, we proceed by defining the limiting sublevel sets and limiting normal operator maps, useful tools to study quasiconvex functions. 
For a proper function $g: \mathbb{R}^{n} \rightarrow \overline{\mathbb{R}}$ and $x \in {\rm dom}\,h$, the sublevel set of $g$ at $g(x)$, the strict sublevel of $g$ at $g(x)$ and the {\it limiting sublevel} of $g$ at $g(x)$ are defined by $S_{g} (x) := \{y \in \mathbb{R}^{n}: \, g(y) \leq g (x)\}$, $S^{<}_{g} (x) := \{y \in \mathbb{R}^{n}: \, g(y) < g (x)\}$ and $S^{lim}_{g}(\bar{x}) := \liminf_{y \rightarrow \bar{x}} S_{g}(y)$, respectively.  
 
The {\it normal} operator and {\it limiting normal} operator of $g$ at $g(x)$ defined by  
\begin{align*}
 N_{g} (x) := (S_{g} (x)- x)^{\circ} ~~ \text{ and } ~~ N_{g}^{lim} (x) := (S^{lim}_{g}(x) - x)^{\circ}. 
\end{align*} 

Finally, we mention other useful subdifferentials suitable for the study of optimality conditions for quasiconvex mappings. 
We start with the {\it Greenberg-Pierskalla} subdifferential $\partial^{GP}h$, given by 
   \begin{equation}
\partial^{GP} h(x) := \{ v \in \mathbb{R}^{n}: \text{ if } \langle v, y-x \rangle \geq 0 \implies h(y) \geq h(x)\}.
  \end{equation}
For instance, it holds that $0\in \partial^{GP}h(\bar{x})$ whenever $\bar{x}$ is a global minimizer of $h$. 
For more details and applications, see \cite{suzuki, yaoeuro} and references therein. 
  
Another important subdifferential is the {\it quasiconvex subdifferential} $\partial^{q} h$, defined for all $x\in \text{dom} h$ as follows:
    \begin{align}
   \partial^{q}h(x):=
       \left\{
            \begin{array}{cr}
\partial^{FM}h(x)\cap N(S_{h}(x), x), & \text{if }  N(S^{<}_{h}(x), x)\neq \{0\};  \\
\emptyset, & \text{if }  N(S^{<}_{h}(x), x)=\{0\};  \\ 
            \end{array}
       \right.
    \end{align}
It has many important properties, it is always cyclically quasimonotone and coincides with Fenchel-Moreau subdifferential whenever the function is convex. 
See \cite{appropriatesubqc} for more results about the quasiconvex subdifferential. 

\section{Properties of the Strong Subdifferential}\label{sec:03} 

Recently, to analyze the behavior of quasiconvex functions and the relations with the proximal point methods in the nonconvex case, a new notion of subdifferential had been considered the so-called {\it strong subdifferential}, see \cite{Kab-Lara-2}. 



\begin{definition}\label{def:strongdifferentiable}
Given a nonempty set $K \subseteq \mathbb{R}^{n}$, $\beta > 0$ and $\gamma \geq 0$. The  $(\beta, \gamma, K)$-strong subdifferential of $h$ at $\overline{x} \in \mathrm{dom}\,h\cap K$ is the set $\partial^{K}_{\beta, \gamma} h(\overline{x}) $ defined by the following:

A vector $\xi \in \partial^{K}_{\beta, \gamma} h(\overline{x})$ if 
   \begin{align}\label{def:subd}
\max\{h(y), h(\overline{x})\} \geq h(\overline{x}) + \frac{\lambda}{\beta} \langle \xi, y - \overline{x}\rangle + \frac{\lambda}{2} \left(\gamma - \frac{\lambda}{\beta} - \lambda \gamma \right) \lVert y - \overline{x} \rVert^{2}, 
  \end{align}
for every $y \in K$ and for every $\lambda \in [0, 1]$.   

If $K = S_{h}(\overline{x})$, we define the $(\beta, \gamma)$-SS (strong sublevel) subdifferential of $h$ by 
  \begin{align} \label{eq:SS}
  \partial_{\beta, \gamma} h(\overline{x}) := \{\xi \in \mathbb{R}^{n} \mid \, \langle \xi, y - \overline{x}\rangle \leq & - \frac{\beta \gamma}{2} \lVert y - \overline{x} \rVert^{2}, ~ \forall ~ y \in S_{h} (\bar{x})\}. 
 \end{align} 
\end{definition}

We would like to mention that if $\gamma > 0$, the subdifferential $\partial^{K}_{\beta, \gamma} h(\overline{x})$ is motivated for strongly quasiconvex functions, while if $\gamma = 0$, then $\partial^{K}_{\beta, 0} h(\overline{x})$ is motivated for quasiconvex functions. 
Note that the strong subdifferential may provide useful information on strongly quasiconvex functions but other subdifferentials  not, \cite[Remark 20]{Kab-Lara-2}. 


Further properties for the strong subdifferential are listed below (see \cite{Kab-Lara-2}):
 \begin{itemize}

 \item[$(P1)$] For $x \!\in\! {\rm dom}\,h \!\cap\! K$, $\beta \!>\! 0, \gamma \!\geq \!0$ we have $\partial^{K}_{\beta, \gamma} h(x)$ is closed and convex.
  

 \item[$(P2)$] $\partial^{K}_{\beta, \gamma} h(\overline{x})$ is compact for 
 all $\overline{x} \in {\rm int\,(dom}\,h\cap K{\rm )}$ by \cite[Proposition 7$(d)$]{Kab-Lara-2}.



  \item[$(P3)$] If $h$ is strongly quasiconvex with modulus $\gamma > 0$ and lsc, then $\partial^{K}_{\beta, \gamma} h(x) \neq \emptyset$ for all $x \in {\rm dom}\,h \cap K$ by \cite[Corollary 38$(a)$]{Kab-Lara-2}.

  \item[$(P4)$] If $h$ is quasiconvex such that $ \liminf_{\lVert x \rVert \rightarrow + \infty} \frac{h(x)}{\lVert x \rVert^{2}} \geq 0$, then $\partial^{K}_{\beta, \gamma} h(x) \neq \emptyset$ for all $x \in {\rm dom}\,h \cap K$ by \cite[Corollary 38$(b)$]{Kab-Lara-2}.

  \item [$(P5)$] If $x \in {\rm dom}\,h$ and $S_{h(x)} (h) \subseteq K$ then $\partial^{K}_{\beta, \gamma} h(x) \, \subset \, \partial_{\beta, \gamma} h(x)$ for all $\beta > 0$, $\gamma \geq 0$. 

   \item [$(P6)$] If $K_{1} \subset K_{2}$, then  $\partial^{K_{2}}_{\beta, \gamma} h(x) \subset \partial^{K_{1}}_{\beta, \gamma} h(x)$.
 \end{itemize}
 See \cite{Kab-Lara-2} for further properties on the strong subdifferentials.

\subsection{Relation with other subdiferentials} 
\label{subsec:regularhorizonsubdiferential}

We discuss some relation between the regular, the limiting and the horizon subdifferentials with the strong subdifferential. 
Relations with the Greenberg-Pierskalla and quasiconvex subdifferentials are also discussed. 

\begin{proposition}\label{lem:inclusion}
 Let $h$ be a proper function, $\bar{x} \in {\rm dom}\,h$ and $V$ be a convex neighborhood of $\bar{x}$. Suppose that $h$ is strongly quasiconvex function with modulus $\gamma \geq 0$ on $V$. If $v \in \widehat{\partial} h(\bar{x})$, then $v \in \partial^{K}_{1, \gamma} h(\overline{x})$, where $K := S_{h}(\bar{x}) \cap V$.
\end{proposition}
    \begin{proof}
Take $v \in \widehat{\partial}h(\bar{x})$. Then $h(y)\geq h(\bar{x})+\langle v, y-\bar{x} \rangle +o(\|y-\bar{x}\|)$ for all $y$. 
Since $h$ is strongly quasiconvex on $V$, we have
\begin{equation*}
  h(\lambda z + (1-\lambda)x) \leq \max \{h(z), h(x)\} - \lambda(1-\lambda) \frac{\gamma}{2} \lVert z - x \rVert^{2}, ~ \forall ~ z, x \in V, \lambda\in [0,1].
 \end{equation*}
Take $y \in K = S_{h} (\bar{x}) \cap V$ and $t \in (0, 1]$. 
Set $z_{t}:=(1-t)\bar{x}+ty \in V$. Thus,
\begin{align*}
\langle v, z_{t}-\bar{x} \rangle +o(\|z_{t}-\bar{x}\|) +h(\bar{x}) \leq  h(z_{t}) \leq \max \{h(y), h(\bar{x})\} - t(1 -  t) \frac{\gamma}{2} \lVert y - x \rVert^{2}.
 \end{align*}

As $h(y) \leq h(\bar{x})$, the above yields $t \langle v, y-\bar{x} \rangle + o (t\|y-\bar{x}\|) \leq - t(1 - t) \frac{\gamma}{2} \lVert y - x \rVert^{2}$. Thus, dividing by $t>0$ and taking limit when $t\rightarrow 0^{+}$, we get 
\begin{align*}
 \langle v, y - \bar{x} \rangle \leq -\frac{\gamma}{2} \lVert y - \bar{x} \rVert^{2}, ~ \forall ~ y \in S_{h(\bar{x})}(h) \cap V.
\end{align*}
The above yields $v \in \partial^{K}_{1, \gamma} h(\overline{x})$ with $K := S_{h}(\bar{x}) \cap V$.
\end{proof}

Following the proof of Proposition \ref{lem:inclusion}, it is not difficult to show the next statement.
  \begin{proposition}\label{propo:convexq}
      We always have  $\partial^{q}h(\bar{x}) \subset \partial_{1, \gamma}^{S_{h}(\bar{x})\cap V}h(\bar{x})$, if $h$ is strongly quasiconvex with modulus $\gamma \geq 0$ over a convex set $V$. 
  \end{proposition}

For general convex sets, the above subdifferentials can be very different. For example, consider the strongly quasiconvex function $h(x)=\sqrt{|x|}$.
In this case, from \cite{Kab-Lara-2}, one has $\partial_{\beta, \frac{1}{2}}^{[-1, 1]}h(0)=[-\beta-\frac{1}{2}, \beta+\frac{1}{2}]$, and by \cite{appropriatesubqc}, we get $\partial^{q}h(0)=\mathbb{R}$.
On the other hand, for the function $g_{1}$ defined in Example \ref{ex:strictnormalcone}, one has $\partial_{1, 1}^{\mathbb{R}}g_{1}(0)=(-\infty, -\frac{1}{2}]$ and $\partial^{q}g_{1}(0)=\emptyset$. 

Now, we proceed by analyzing the relation with the limiting and horizon subdifferential. Since the above subdifferentials take into account the limit of certain sequences, we need to study the limit behaviour of the level sets. 
We have the following proposition. 
         
\begin{proposition}\label{propo:limitinghorizon}
Let $h$ be a proper function, $\bar{x} \in {\rm dom}\,h$ and $V$ be a convex neigh\-bor\-hood of $\bar{x}$ with $\bar{x}\in \text{int }V$. Suppose that $h$ is strongly quasiconvex with modulus $\gamma \geq 0$ on $V$ and that $S_{h} (\cdot) \cap V$ is inner semicontinuous at $\bar{x}$, that is, $S_{h}(\bar{x})\cap V \subset \liminf_{ y \rightarrow \bar{x}, y \in \text{dom}(h)}  S_{h}(y)\cap V$. Then, the following statements hold:
\begin{itemize}
 \item[$(a)$] $\partial h(\bar{x}) \subset \partial^{S_{h}(\bar{x})}_{1, \gamma} h(\overline{x})$; 
    
 \item[$(b)$] $\partial^{\infty} h(\bar{x}) \subset (S_{h} (\bar{x}) \cap V - \bar{x})^{\circ}$.
  \end{itemize}
\end{proposition}

\begin{proof}
 $(a)$ Take $v \in \partial h(\bar{x})$. Then there exist sequences $\{x^{k}\}, \{v^{k}\} \subset \mathbb{R}^{n}$ with $x^{k} \rightarrow_{h} \bar{x}$ and $v^{k} \rightarrow v$ such that $v^{k} \in \widehat{\partial} h(x^{k})$, $\forall k$. 
 As $\bar{x} \in \text{int }V$, for $k$ large enough, we get $x^{k}\in V$.
 Now, set $K_{k} := S_{h}(x^{k}) \cap V$, $\forall k \in \mathbb{N}$.  
 Using Proposition \ref{lem:inclusion}, for $k$ large enough, one has $v^{k} \!\in\! \partial^{K_k}_{1, \gamma} h(x^{k})$. 

By the inner semicontinuity of $S_{h} (\cdot) \cap V$ restricted to $\text{dom }h$, we see that for every $y \in S_{h} (\bar{x}) \cap V$ there exists a sequence $\{y^{k}\} \subset \mathbb{R}^{n}$ with $y^{k} \in S_{h}(x^{k}) \cap V$ and $y^{k}\in \text{dom }h$ such that $y^{k} \rightarrow y$.
Now, we observe that $v^{k} \in \partial^{K_k}_{1, \gamma} h(x^{k})$ for all $k$ implies $\langle v^{k}, y^{k}-x^{k} \rangle \leq -(\gamma/2)\lVert y^{k} - x^{k} \rVert^{2}$ for all $k$. 
By taking limit in the last expression, we have $\langle v, y-\bar{x} \rangle \leq -(\gamma/2) \lVert y-\bar{x} \rVert^{2}$ for all $y \in S_{h} (\bar{x}) \cap V$, that is, $v \in \partial^{K}_{1, \gamma} h(\bar{x})$, thus $(a)$ holds.  
                
The proof for (b) is similar.    
\end{proof}

\begin{proposition}\label{propo:interior}
 Let $h$ be a continuous function at $\bar{x} \in {\rm dom}\,h$ with $h(\bar{x}) < 0$ and set $[h \leq 0] := \{x \in \mathbb{R}^{n}: h(x) \leq 0\}$. Then the following assertions hold:
\begin{itemize}
 \item[$(a)$] $\partial^{[h \leq 0]}_{\beta, \gamma} h(\bar{x}) = \emptyset$ (if $\beta>0$ and $\gamma>0$).
 
 \item[$(b)$] $\partial^{[h \leq 0]}_{\beta, \gamma} h(\bar{x}) = \{0\}$ (if $\beta > 0$ and $\gamma = 0$). 
\end{itemize} 
\end{proposition}

\begin{proof}
The proofs are similar, we just prove $(a)$: By contradiction, assume that $\beta>0$ and $\gamma>0$ and that $\partial^{[h \leq 0]}_{\beta, \gamma} h(\bar{x}) \neq \emptyset$. Take $v \in \partial^{[h \leq 0]}_{\beta, \gamma} h(\bar{x})$. Since $h$ is con\-ti\-nuous at $\bar{x}$, there exists $\varepsilon_{0} > 0$ such that $\bar{x} + \varepsilon w \in [h \leq 0]$ for every $w \in \mathbb{R}^{n}$ with $\|w\|=1$ and every $0 < \varepsilon \leq \varepsilon_{0}$. 
Thus, $\langle v, y- \bar{x} \rangle \leq -(\beta \gamma/{2}) \lVert y - \bar{x} \rVert^{2}$ for all $y \in [h\leq 0]$. Then, by taking $y = \bar{x} + \epsilon w$, we have 
$$\langle v, \varepsilon w \rangle \leq - (\beta \gamma/{2}) \varepsilon^{2} \|w\|^{2} = - (\beta \gamma/{2}) \varepsilon^{2}.$$
As a consequence $\|v\|_{\infty} = \sup_{w: \|w\|=1} \langle v, w \rangle \leq - (\beta \gamma/{2}) \varepsilon < 0$, a contradiction. 
Therefore, $\partial^{[h \leq 0]}_{\beta, \gamma} h(\bar{x}) = \emptyset$.
\end{proof}

We continue with the relation with the Greenberg-Pierskalla subdifferential. We have the following statement.

 \begin{proposition}\label{propo:gpsubd}
We always have $\partial_{\beta, \gamma}^{\mathbb{R}^{n}}h(\bar{x}) \subset \partial^{GP}h(\bar{x})$ for every $\gamma, \beta>0$.     
 \end{proposition}
   \begin{proof}
Take $\xi \in \partial_{\beta, \gamma}^{\mathbb{R}^{n}}h(\bar{x})$.
To show that $\xi \in \partial^{GP}h(\bar{x})$, we suppose that $\langle \xi, y-\bar{x}\rangle \geq 0$ for some $y$. 
Let $\lambda \in (0, \frac{\gamma \beta}{1+\gamma \beta})$ and set $\alpha:=\frac{\lambda}{2}(\gamma-\frac{\lambda}{\beta}-\gamma\lambda)$. This choice of $\lambda$ implies $\alpha>0$.  
Hence, from \eqref{def:subd}, one has 
    \begin{align*}
        \max\{h(y), h(\bar{x})\}\geq h(\bar{x})+\alpha \|y-\bar{x}\|^{2}. 
    \end{align*}
The above expression yields $h(y)\geq h(\bar{x})$. 
   \end{proof}
  The inclusion in the above proposition can be strict, in fact, in Example \ref{ex:strictnormalcone} and using the corresponding notation, we have $\partial_{1, 1}^{\mathbb{R}}g_{1}(\bar{x})=(-\infty, -1/2]$ and $\partial^{GP}g_{1}(\bar{x})=(-\infty, 0)$.
For general convex set $K$, both subdifferentials can be very different. Indeed, from Remark \ref{remark:fhregular}, one has $\partial_{1, 1}^{[0, 1]}g(0)=(-\infty, -1/2]$ meanwhile $\partial^{GP}g(0)=(0, \infty)$. 



\subsection{Basic calculus rules for the strong subdifferential}
\label{subsec:basiccalculusrules}

We continue by providing some basic calculus rules for the strong subdifferential of the supremum of a finite number of functions, an important operation in variational analysis.
We start with the following proposition. 

\begin{proposition}\label{propo:inclusionmax}
 Let $\{g_{j}\}_{j \in I}$ be a finite number of extended-valued functions on $\mathbb{R}^{n}$ and set $g := \sup_{j \in I} g_{j}$. Then, for every $\bar{x} \in {\rm dom}\,g$, we have
\begin{align}\label{eqn:subdmaximum}
 \overline{co \left( \bigcup_{j\in I(\bar{x})} \partial_{\beta, \gamma_{j}}^{K_{j}} g_{j}(\bar{x}) \right)} \subset \partial_{\beta, \gamma_{m}}^{K} g(\bar{x}); 
\end{align} 
 where $I(\bar{x}) := \{j \in I: g(\bar{x}) = g_{j}(\bar{x})\}$, $K \subset \cap_{j \in I(\bar{x})}K_{j}$ and $\gamma_{m} := \min_{j \in I(\bar{x})} \gamma_{j}$.
 \end{proposition}
     \begin{proof}
Since $\bar{x} \in \text{dom }g\cap K$, by \cite[Proposition 7]{Kab-Lara-2} we get that $\partial_{\beta, \gamma_{m}}^{K} g(\bar{x})$ is convex and closed set. So, it suffices to show that if $w \in \bigcup_{j \in I(\bar{x})} \partial_{\beta, \gamma_{j}}^{K_{j}} g_{j} (\bar{x})$, then $w \in \partial_{\beta, \gamma_{m}}^{K} g(\bar{x})$. Indeed, take $j \in I(\bar{x})$ and $w \in \partial_{\beta, \gamma_{j}}^{K_{j}} g_{j}(\bar{x})$. Thus, for every $y \in K_{j}$ and every $\lambda \in [0, 1]$, we have
\begin{align}\label{eqn:wj}
 \frac{\lambda}{\beta}\langle w, y - \bar{x} \rangle \leq \max\{g_{j}(y), g_{j} (\bar{x})\} -\frac{\lambda (1 - \lambda) \gamma_{j}}{2} \|y-\bar{x} \|^{2} + \frac{\lambda^{2}}{2 \beta} \|y-\bar{x}\|^{2}.
\end{align}

Since $j \in I(\bar{x})$, $g_{j} (\bar{x}) = g(\bar{x})$, and as $K \subset \cap_{i\in I(\bar{x})} K_{i}$, $\gamma_{m} \leq \gamma_{j}$, $g_{j} (y) \leq g(y)$, expression \eqref{eqn:wj} implies that
$$\frac{\lambda}{\beta} \langle w, y-\bar{x} \rangle \leq \max\{g(y), g(\bar{x})\} - \frac{\lambda (1-\lambda) \gamma_{m}}{2} \|y-\bar{x}\|^{2} + \frac{\lambda^{2}}{2 \beta} \|y - \bar{x}\|^{2},$$ 
for all $y \in K$ and all $\lambda \in [0, 1]$, thus $w \in \partial_{\beta, \gamma_{m}}^{K} g(\bar{x})$.
\end{proof}

The inclusion in \eqref{eqn:subdmaximum} may be strict even in some simple situations.  

\begin{example}(The inclusion in \eqref{eqn:subdmaximum} may be strict).
Consider the extended-valued functions $g_{1}: \mathbb{R} \rightarrow \mathbb{R}$ and $g_{2}: \mathbb{R} \rightarrow \mathbb{R}$ defined as 
\begin{align*}
 g_{1} (x) = \left\{
 \begin{array}{ll}
  0, & {\rm if} ~ x \geq 0,\\
  \frac{1}{2} x^{2}, & {\rm if} ~ x < 0.
  \end{array}
 \right.
 \text{ and }\
 g_{2} (x) = \left\{
 \begin{array}{cl}
  \frac{1}{2}x^{2} & {\rm if} ~ x \geq 0,  \\
  0 & {\rm if} ~ x<0. \\
 \end{array}
 \right.  
\end{align*} 

Set $g := \max\{g_{1}, g_{2}\}$, $K := K_{1} = K_{2} = \mathbb{R}$, $\gamma = \gamma_{1} = \gamma_{2}=1$ and $\beta=1$. Take $\bar{x} = 0$. Then, $\partial_{1, 1}^{K} g_{1} (\bar{x}) = \partial_{1, 1}^{K} g_{1} (\bar{x}) = \emptyset$, and since $g(x) = \frac{1}{2}x^{2}$ is strongly quasiconvex, by using $(P3)$ we get $\partial_{1, 1}^{K} g(\bar{x}) \neq \emptyset$. 
Hence, the inclusion in \eqref{eqn:subdmaximum} is strict. 
\end{example}

By Proposition \ref{propo:inclusionmax}, we see that there is some freedom in the choice of $K_{j}$ meanwhile $K\subset \cap_{j \in I(\bar{x})}K_{j}$. 
A particular case choice is $K_{j}:=S_{g_{j}}(\bar{x})$, $\forall j$ and $K=S_{g}(\bar{x})$. Clearly, $K\subset \cap_{j \in I(\bar{x})}K_{j}$ and as a consequence of Proposition \ref{propo:inclusionmax}, we get
\begin{align*}
 \overline{ co \left(\bigcup_{j \in I(\bar{x})} \partial_{\beta, \gamma_{j}} g_{j} (\bar{x}) \right)} \subset \partial_{\beta, \gamma_{m}} g(\bar{x}). 
\end{align*} 

Now, we focus on a special case where \eqref{eqn:subdmaximum} holds as equality. 
First, let us recall some definitions (see \cite[Section 3]{Kab-Lara-2}). 
\begin{definition}\label{def:fregular}
 Let $g: \mathbb{R}^{n} \rightarrow \overline{\mathbb{R}}$ be a proper function, $K \subset \mathbb{R}^{n}$, $\beta>0$, $\gamma \geq 0$ and $x \in ({\rm dom}\,g) \cap K$. 
 The upper Dini directional derivative of $g$ at $x \in {\rm dom}\,g$ in the direction $d \in \mathbb{R}^{n}$ is defined by 
 \begin{align*}
 g^{D+} (x; d) := \limsup_{t \rightarrow 0^{+}} \frac{g(x + td) - g(x)}{t}.
\end{align*}
Furthermore, we say that $g$ is $F$-regular at $x$ on $(K, \beta, \gamma)$, if for every $d \in \mathbb{R}^{n}$, we have 
$$g^{D+} (x; d) \geq 0 \, \implies \, \sigma (\partial_{\beta, \gamma}^{K} g(x); d) \geq 0.$$
\end{definition}

For more results of the above definition, see \cite[Section 3]{Kab-Lara-2}. 
We proceed with the next result. 

\begin{proposition}\label{propo:reverseinclusionmax}
Let $\{g_{j}\}_{j \in I}$ a finite number of extended-valued functions on $\mathbb{R}^{n}$, $g:=\sup_{j \in I} g_{j}$, $\bar{x} \in {\rm dom}\,g$ and $K$ be a convex set such that $K \subset \cap_{j\in I(\bar{x})}K_{j}$. 
Assume that $\cap_{j \in I(\bar{x})}[\text{cone } \partial_{\beta, \gamma_{j}}^{K_{j}} g_{j}(\bar{x})]^{\circ} \subset \text{cone}(K-\bar{x})$ and that $\partial_{\beta, \gamma_{j}}^{K_{j}} g_{j} (\bar{x})$ is compact and $g_{j}$ is F-regular at $\bar{x}$ for all $j \in I(\bar{x})$. 

Furthermore, we assume that the following Slater-type condition holds:
\begin{align}\label{eqn:slatertype}\tag{$\mathbb{S}$}
 \text{ there exists }\widehat{d}\in D(K, \bar{x}) \text{ such that } \sigma( \partial_{\beta, \gamma_{j}}^{K_{j}} g_{j}(\bar{x}), \widehat{d}) < 0 \text{ for all } j\in I(\bar{x}). 
\end{align}
Then
\begin{align}\label{eqn:subdmaximumreverse}
 \overline{co \left(\bigcup_{j\in I(\bar{x})} \partial_{\beta, \gamma_{j}}^{K_{j}} g_{j}(\bar{x}) \right)} = \partial_{\beta, \gamma_{m}}^{K} g(\bar{x}). 
\end{align} 
\end{proposition}

\begin{proof}
Since \eqref{eqn:subdmaximum} holds, we show only the reverse inclusion. Indeed, suppose for the contrary that there exists $v \in \partial_{\beta, \gamma_{m}}^{K} g(\bar{x})$ with $v \notin \overline{co \left(\bigcup_{j \in I(\bar{x})} \partial_{\beta, \gamma_{j}}^{K_{j}} g_{j}(\bar{x})\right)}$.
By the Hahn-Banach theorem, there exists $d\neq 0$ satisfying 
$$\langle v, d \rangle >0 \geq \langle w, d\rangle, ~ \forall ~ w \in \bigcup_{j\in I(\bar{x})} \partial_{\beta, \gamma_{j}}^{K_{j}} g_{j}(\bar{x}),$$ 
which implies that 
$d \in \cap_{j \in I(\bar{x})}[\text{cone }\partial_{\beta, \gamma_{j}}^{K_{j}}g_{j}(\bar{x})]^{\circ}$, and hence $d\in \text{cone}(K-\bar{x})$.

Take $\widehat{d}$ satisfying condition \eqref{eqn:slatertype}. Then taking $d_{\theta} := d+\theta \widehat{d}$ with $\theta > 0$ small enough, we get    
\begin{align}\label{eqn:hbsupremum}
\langle v, d_{\theta} \rangle >0 > \langle w, d_{\theta} \rangle, \ \ \forall ~ w \in \bigcup_{j\in I(\bar{x})} \partial_{\beta, \gamma_{j}}^{K_{j}} g_{j}(\bar{x}).
\end{align}
Since $K$ is convex and $\hat{d}\in D(K, \bar{x})=\mathbb{R}_{+}(K-\bar{x})$, one has $d_{\theta}=d+\theta \hat{d}\in D(K, \bar{x})$. From \eqref{eqn:hbsupremum} and by the compactness of $\partial_{\beta, \gamma_{j}}^{K_{j}} g_{j}(\bar{x})$, we have $0 > \sigma (\partial_{\beta, \gamma_{j}}^{K_{j}} g_{j}(\bar{x}), d_{\theta})$ for all $j \in I(\bar{x})$. 

Now, since $I$ is finite, we can find $j_{0} \in I(\bar{x})$ such that $g^{D+} (\bar{x}; d_{\theta}) \leq g_{j_{0}}^{D+} (\bar{x}; d_{\theta})$. Since $d_{\theta} \in  D(K, \bar{x})$ and by \cite[Proposition 9]{Kab-Lara-2}, we have 
  \begin{align*}
 \langle v, d_{\theta} \rangle \leq \beta \max\{g^{D+} (\bar{x}; d_{\theta}), 0\}.     
  \end{align*}

From \eqref{eqn:hbsupremum}, we get $0 < \langle v, d_{\theta} \rangle$, and thus $0 < g^{D+} (\bar{x}; d_{\theta}) \leq g_{j_{0}}^{D+}(\bar{x}; d_{\theta})$. Finally, since $g_{j_{0}}$ is $F$-regular at $\bar{x}$, it follows that $\sigma (\partial_{\beta, \gamma_{j_{0}}}^{K_{j_{0}}} g_{j_{0}}(\bar{x}), d_{\theta})\geq 0$, which is a contradiction. 
As consequence, $v \in \overline{co \left(\bigcup_{j \in I(\bar{x})} \partial_{\beta, \gamma_{j}}^{K_{j}} g_{j}(\bar{x})\right)}$ and the proof is complete.
\end{proof}

\section{Normal Cone of a Level set}\label{sec:04}
Let $\{g_{j}\}_{j \in I}$ be a finite number of extended-valued functions on $\mathbb{R}^{n}$ and set $g:=\sup_{j \in I} g_{j}$. Furthermore, define
\begin{align*}
 \Omega := \left\{x \in \mathbb{R}^{n} \mid \, g(x) = \sup_{j \in I} g_{j} (x) \leq 0 \right\}. 
\end{align*}
In this section, we focus on obtaining a computable expression for the normal cone of $\Omega$ at a feasible point $\bar{x} \in \Omega$. 
Our main assumption about the feasible set is that 
   \begin{align*}
       \Omega \ \text{ is convex}. 
   \end{align*}
   
Our approach is very general since we only suppose that $\Omega$ is convex, and we do not impose any kind of locally Lipschitz, continuity or convexity assumption on the functions $\{g_{j}\}_{j \in I}$. We mention that the convexity of $\Omega$ holds if $g$ is quasiconvex (in particular if the function $g$ is convex). 

The next proposition provides a lower estimate of the normal cone in terms of the strong subdifferentials and the normal operators of the functions $g_{j}$. 

\begin{proposition}\label{propo:normalcone}
 Let $\{g_{j}\}_{j \in I}$ be a finite number of extended-valued functions in $\mathbb{R}^{n}$, $g := \sup_{j \in I} g_{j}$, $\bar{x} \in {\rm dom}\,g \cap \Omega$ and $\{K_{j}\}_{j\in I}$ be a family of sets such that $\Omega \subset \cap_{j \in I(\bar{x})} K_{j}$ where $I(\bar{x}) := \{j \in I\mid \, g(\bar{x}) = 0\}$. 
 Then 
\begin{align}\label{eqn:normalconemax}
 \overline{\bigcup_{\mu \in \mathbb{R}_{+}^{ |I(\bar{x})|}} \left\{ \sum_{j \in I(\bar{x}): \mu_{j} > 0} \mu_{j} \partial_{\beta_{j}, \gamma_{j}}^{K_{j}} g_{j} (\bar{x}) + \sum_{j \in I(\bar{x}): \mu_{j} = 0} N_{g_{j}} (\bar{x}) \right\} } \subset N(\Omega, \bar{x}), 
\end{align}  
 \end{proposition}
     \begin{proof}
Since $N(\Omega, \bar{x})$ is a closed convex cone, it suffices to show that $\partial_{\beta_{j}, \gamma_{j}}^{K_{j}} g_{j}(\bar{x})$ and $N_{g_{j}}(\bar{x})=(S_{g_{j}}(\bar{x})-\bar{x})^{\circ}$ are subsets of $N(\bar{x}, \Omega)$ for all $j \in I(\bar{x})$. Since $\Omega \subset S_{g_{j}}(\bar{x})$ for every $j \in I(\bar{x})$, we have  $(S_{g_{j}}(\bar{x})-\bar{x})^{\circ} \subset  N(\Omega, \bar{x})$. 

Take $w \in \partial_{\beta_{j}, \gamma_{j}}^{ K_{j}} g_{j} (\bar{x})$ for some $j \in I(\bar{x})$. Then, for every $y \in K_{j}$,
   \begin{align}\label{eqn:wicone}
 \frac{\lambda}{\beta_{j}} \langle w, y - \bar{x} \rangle \leq \max\{g_{j} (y), g_{j} (\bar{x})\} - \frac{\lambda (1 - \lambda) \gamma_{j}}{2} \| y - \bar{x}\|^{2} + \frac{\lambda^{2}}{2 \beta_{j}} \| y - \bar{x}\|^{2}. 
   \end{align}

 Since $\Omega \subset S_{g_{j}} (\bar{x})$ and $\Omega \subset K_{j}$, \eqref{eqn:wicone} implies $\langle w, y-\bar{x} \rangle \leq -\frac{1}{2} \beta_{j} \gamma_{j}\|y-\bar{x}\|^{2}\leq 0$ for all $y \in \Omega$. As a consequence, $w \in  N(\Omega, \bar{x})$.  
\end{proof}

The lower estimative of the normal cone given by \eqref{eqn:normalconemax} may be strict as the following example shows. 

\begin{example}\label{ex:strictnormalcone}(The inclusion in \eqref{eqn:normalconemax} may be strict).
Consider the extended-valued functions $g_{1}:\mathbb{R} \rightarrow \overline{\mathbb{R}}$ and $g_{2}:\mathbb{R} \rightarrow \overline{\mathbb{R}}$ defined as 
\begin{align*}
 g_{1} (x) = \left\{
 \begin{array}{cl}
  0 & {\rm if} ~ x=0,  \\
  -x^{-1} & {\rm if} ~ 0 < x \leq 1, \\
  + \infty & {\rm otherwise},
  \end{array}
  \right.
  \text{ and }\
  g_{2}(x) = \left\{
  \begin{array}{cl}
  0 & {\rm if} ~ x=0, \\
  + \infty  & {\rm if} ~ 0 < x \leq 1, \\
  - x^{-1} & {\rm otherwise}.
  \end{array}
  \right.  
 \end{align*} 
 
Set $K_{1} = K_{2} = \mathbb{R}$, $\gamma= \gamma_{1} = \gamma_{2} = 1$ and $\beta=1$. Take $\bar{x} = 0$. Then, by direct calculation, we have $\partial_{1, 1}^{\mathbb{R}} g_{1}(\bar{x}) = (-\infty, -1/2]$ and $\partial_{1, 1}^{\mathbb{R}} g_{2} (\bar{x}) = \emptyset$, while $N_{g_{1}} (0) = [0,1]^{\circ} = \mathbb{R}_{-}$ and $N_{g_{2}} (0) = (\{0\} \cup (1, \infty))^{\circ} = \mathbb{R}_{-}$.

On the other hand, since $g := \max\{g_{1}, g_{2}\}$, we get $g(0)=0$ and $g(x)=+\infty$, $\forall x \neq 0$. 
Thus, $\Omega=\{0\}$, $N(\Omega, \bar{x})=\mathbb{R}$ and the inclusion in \eqref{eqn:normalconemax} is strict. 
\end{example}

In order to obtain the equality in \eqref{eqn:normalconemax} under mild assumptions, we first study some properties of the directional derivatives. To that end, we define the {\it upper Hadamard directional derivative} of $g$ at $x \in {\rm dom}\ g$ in the direction $d \in \mathbb{R}^{n}$ by  
\begin{align*}
 g^{H+} (x; \bar{d}) := \limsup_{t \rightarrow 0^{+}, \ d\rightarrow \bar{d}} \frac{g(x +td) - g(x)}{t}.
\end{align*}
Clearly, $g^{H+} (x; \bar{d})$ is well-defined if $g$ is locally Lipschitz at $x$. Certainly, for every $d \in \mathbb{R}^{n}$ one has $g^{D+} (x; \bar{d}) \leq g^{H+} (x; \bar{d})$.
The next technical lemma can be seen as a complement to \cite[Proposition 9]{Kab-Lara-2}. 
The proof follows from \cite[Proposition 9]{Kab-Lara-2} with the appropriate modifications. 
  
\begin{lemma}\label{lem:derivative}
 Let $\bar{x} \in {\rm dom}\,g$ and $\partial_{\beta, \gamma}^{K} g(\bar{x})$ be the strong subdifferential for some $K$, $\beta>0$ and $\gamma \geq 0$. Then, for every $d \in T(K, \bar{x})$, we have
   \begin{align*}
   \lambda \langle w, d \rangle \leq \beta \max\{g^{H+} (\bar{x}; d), 0\}, 
   \ \ \forall \lambda \in [0, 1] \text{ and }\forall w \in \partial_{\beta, \gamma}^{K} g(\bar{x}).
   \end{align*}
\end{lemma}

Finally, consider the next definition. 

\begin{definition}\label{def:fhregular}
 Let $g: \mathbb{R} \rightarrow \overline{\mathbb{R}}$ be a proper function, $K \subset \mathbb{R}^{n}$, $\beta>0$, $\gamma \geq 0$ and $x \in {\rm dom}\,g \cap K$. We say that $g$ is $F_{H}$-regular at $x$ on $(K, \beta, \gamma)$, if for every $d \in \mathbb{R}^{n}$, the following implication holds
 \begin{align*}
  g^{H+} (x; d)\geq 0 \implies \sigma (\partial_{\beta, \gamma}^{K} g(x);d) \geq 0.
 \end{align*}
\end{definition}

\begin{remark}\label{remark:fhregular}
    We point out that a function can be $F_{H}$-regular for some $K$ but not for other sets. In fact, consider $g:\mathbb{R}\rightarrow \mathbb{R}$ defined as 
\begin{align*}
 g(x) = \left\{
 \begin{array}{cl}
  0 & {\rm if} ~ x \geq 0,  \\
 -x & {\rm if } ~ x < 0.
 \end{array}
 \right. 
\end{align*} 

Take $\bar{x} =0$. Note that $g^{H+} (0;d) = 0$ if $d\geq 0$,  and $g^{H+} (0;d) = -d$ if $d < 0$. 
Now, take $K_{1}=[-1, 0]$ and $\beta = \gamma=1$, then $\partial_{\beta, \gamma}^{K_{1}} g(0) = [-1, \infty)$ and hence $\sigma (\partial_{1, 1}^{K_{1}} g(0);d) \geq 0$ for all $d$. Thus, $g$ is $F_{H}$-regular at $\bar{x}$ on $(K_{1}, 1, 1)$.

On the other hand, if $K_{2}=[0, 1]$, then $\partial_{1, 1}^{K_{2}} g(0)=(-\infty, -1/2]$ and hence $\sigma(\partial_{1, 1}^{K_{2}} g(0);d)<0$ when $d>0$. 
Thus, $g$ is not $F_{H}$-regular at $\bar{x}$ on $(K_{2}, 1, 1)$.
\end{remark}

We continue with the main result of this section that characterizes the normal cone of $\Omega$ by using strong and horizon subdifferentials. 

\begin{theorem}\label{theo:normalcone}
Let $\{g_{j}\}_{j \in I}$ be a finite number of extended-valued functions on $\mathbb{R}^{n}$,  $g := \sup_{j \in I} g_{j}$, $\bar{x} \in {\rm dom}\,g$, $\{K_{j}\}_{j}$ be a family of closed convex sets, $\beta_{j}>0$ and $\gamma_{j}\geq 0$ for all $j \in I$. Suppose that $\Omega \subset \cap_{j \in I(\bar{x})}K_{j}$ and that the following statements are satisfied:
\begin{enumerate}
 \item[$(a)$] The Slater-type condition \eqref{eqn:slatertypecone} holds at $\bar{x}$ for $\partial_{\beta_{j}, \gamma_{j}}^{K_{j}} g_{j}(\bar{x})$ for all $j$, that is, 
 \begin{align}\label{eqn:slatertypecone}\tag{$S_{N}$}
 \text{there exists }\widehat{d}\in \cap_{j \in I(\bar{x})} T(K_{j}, \bar{x}) \text{ such that } \sigma(\partial_{\beta_{j}, \gamma_{j}}^{K_{j}} g_{j}(\bar{x}), \widehat{d})<0 \ \forall j \in I(\bar{x}). 
 \end{align}

 \item[$(b)$] The next inclusion holds 
 \begin{align*}
 \cap_{j \in I(\bar{x})}[\text{cone }\partial^{\infty} g_{j}(\bar{x})]^{\circ} \cap \cap_{j \in I(\bar{x})}[\text{cone }\partial_{\beta_{j}, \gamma_{j}}^{K_{j}}g_{j}(\bar{x})]^{\circ}\subset \cap_{j \in I(\bar{x})} T(K_{j}, \bar{x}).   
 \end{align*}

 \item[$(c)$] For every $j \in I(\bar{x})$, 
 the inclusion $\partial^{\infty}g_{j}(\bar{x})\subset (S_{g_{j}}(\bar{x})-\bar{x})^{\circ}$ holds, $g_{j}$ is upper semicontinuous (usc) at $\bar{x}$ and $g_{j}$ is $F_{H}$-regular at $\bar{x}$ on $(K_{j}, \beta_{j}, \gamma_{j})$.

 \item [$(d)$]  For every $j \in I(\bar{x})$, one of the following conditions holds:
      \begin{itemize}
          \item [$(d.1)$] the strong subdifferential $\partial_{\beta_{j}, \gamma_{j}}^{K_{j}} g_{j}(\bar{x})$ is compact;
          \item [$(d.2)$] the $\text{cone}(\partial_{\beta_{j}, \gamma_{j}}^{K_{j}} g_{j}(\bar{x}))$ is closed and $0 \notin \partial_{\beta_{j}, \gamma_{j}}^{K_{j}} g_{j}(\bar{x})$.
      \end{itemize}
\end{enumerate}
Then 
\begin{align}\label{eqn:hullnormalconemaxequality}
 N(\Omega, \bar{x}) = \overline{ \text{co } \left(
 \bigcup_{\mu \in \mathbb{R}_{+}^{|I(\bar{x})|}} 
 \left\{\sum_{j\in I(\bar{x}): \mu_{j}>0} \mu_{j} \partial_{\beta_{j}, \gamma_{j}}^{K_{j}} g_{j}(\bar{x}) + \sum_{j\in I(\bar{x}): \mu_{j}=0} \partial^{\infty}g_{j}(\bar{x}) \right\} \right)},   
 \end{align}  
where $I(\bar{x}) := \{j \in I: \,g_{j}(\bar{x}) = 0\}$. \end{theorem}
    \begin{proof}
First, since $\partial^{\infty} g_{j}(\bar{x}) \subset (S_{g_{j}} (\bar{x}) - \bar{x})^{\circ}$ for all $j \in I(\bar{x})$, the inclusion described in \eqref{eqn:normalconemax} holds.
As $N(\Omega, \bar{x})$ is a closed convex set, the right-hand side of \eqref{eqn:hullnormalconemaxequality} is included in $N(\Omega, \bar{x})$. 

Let us prove the equality in \eqref{eqn:hullnormalconemaxequality} under the above hypotheses. Suppose by contradiction that there exists $v \in N(\Omega, \bar{x})$ that does not belong to the right side of \eqref{eqn:hullnormalconemaxequality}. 
Thus, by the Hahn-Banach theorem, there exists $d \neq 0$ satisfying 
$$\langle v, d \rangle > 0 \geq \langle w, d \rangle, ~ \forall ~ w \in \bigcup_{j \in I(\bar{x})} \partial_{\beta, \gamma_{j}}^{K_{j}} g_{j} (\bar{x}) \cup \partial^{\infty} g_{j} (\bar{x}).$$

The latter implies $d \in \cap_{j \in I(\bar{x})} [\text{cone } \partial^{\infty} g_{j} (\bar{x})]^{\circ} \cap \cap_{j \in I(\bar{x})}[\text{cone } \partial_{\beta_{j}, \gamma_{j}}^{K_{j}} g_{j} (\bar{x})]^{\circ}$. So, by (b), we have $d \in \cap_{j \in I(\bar{x})} T(K_{j}, \bar{x})$.

From the Slater-type condition \eqref{eqn:slatertypecone} and due to the convexity of $K_{j}$, $\forall j$; we can find $d_{\theta}:=d+\theta \widehat{d} \in  \cap_{j \in I(\bar{x})}T(K_{j}, \bar{x})$ for some $\theta>0$ small enough such that 
\begin{align}\label{eqn:hbsupremumcone}
\langle v, d_{\theta} \rangle >0 > \langle w, d_{\theta} \rangle,  \ \ \forall ~ w \in \bigcup_{j \in I(\bar{x})}\partial_{\beta_{j}, \gamma_{j}}^{K_{j}} g_{j}(\bar{x}).    
\end{align}

 Now, we claim that $\bar{x} + rd_{\theta} \in \Omega$ for some $r>0$ small enough. Indeed, suppose by contradiction that there exist $j \in I$ and a sequence $\{r_{k}\} \subset \mathbb{R}_{++}$ with $r_{k} \rightarrow 0^{+}$ such that $g_{j}(\bar{x} + r_{k} d_{\theta}) > g_{j} (\bar{x})$. 
  If $j \notin I(\bar{x})$, then the upper semicontinuity of $g_{j}$ at $\bar{x}$ ensures that $g_{j}(\bar{x} + rd_{\theta})<0$ for every $r$ sufficiently small, a contradiction and thus $j \in I(\bar{x})$. 
  
 We point out that $g_{j}(\bar{x} + r_{k} d_{\theta}) > g_{j} (\bar{x})$ implies that $g_{j}^{H+}(\bar{x}, d_{\theta})\geq \limsup \frac{1}{r_{k}}(g_{j}(\bar{x} + r_{k} d_{\theta})-g_{j} (\bar{x}))\geq 0$.
 As $g_{j}$ is $F_{H}$-regular at $\bar{x}$, we have $\sigma(\partial_{\beta_{j}, \gamma_{j}}^{K_{j}} g_{j} (\bar{x}), d_{\theta})\geq 0$.
 Furthermore, as consequence of  \eqref{eqn:hbsupremumcone}, we get $\sigma(\partial_{\beta_{j}, \gamma_{j}}^{K_{j}} g_{j} (\bar{x}), d_{\theta})=0$.
 
 Now, we split the proof depending if for such $j$, the statement ($d.1$) or $(d.2)$ holds. 
    \begin{itemize}
        \item If ($d.1$) holds then, in this case, the compactness of $\partial_{\beta_{j}, \gamma_{j}}^{K_{j}} g_{j} (\bar{x})$ and \eqref{eqn:hbsupremumcone} imply that $\sigma(\partial_{\beta_{j}, \gamma_{j}}^{K_{j}} g_{j} (\bar{x}), d_{\theta})<0$ which contradicts $\sigma(\partial_{\beta_{j}, \gamma_{j}}^{K_{j}} g_{j} (\bar{x}), d_{\theta})\geq 0$.
        \item If ($d.2$) is fulfilled then $P:=\text{cone}(\partial_{\beta_{j}, \gamma_{j}}^{K_{j}} g_{j}(\bar{x}))$ is a pointed closed convex cone. 
        By \cite[Exercise 6.22]{rwets}, the latter implies that $P^{\circ}$ has a nonempty interior.
        Furthermore, \cite[Exercise 6.22]{rwets} says that $\bar{d} \in \text{int }P^{\circ}$ if and only if 
        $\langle \bar{v}, \bar{d}\rangle <0$, $\forall \bar{v} \in P$ with $\bar{v}\neq 0$.
        From above and by \eqref{eqn:hbsupremumcone}, we get $d_{\theta} \in \text{int }P^{\circ}$. So, there exists $\delta>0$ such that $d_{\theta}+\varepsilon \hat{v} \in P^{\circ}$, for all $\varepsilon \in [0, \delta]$ and $\|\hat{v}\|\leq 1$. 
        Hence, we have 
            \begin{align}\label{eqn:conecone}
                \langle w, d_{\theta}+\varepsilon \hat{v}\rangle \leq 0, \ \forall w \in P=\text{cone}(\partial_{\beta_{j}, \gamma_{j}}^{K_{j}} g_{j}(\bar{x})), \ \forall \hat{v} \text{ with }\|\hat{v}\|\leq 1.
            \end{align}
        As $\sigma(\partial_{\beta_{j}, \gamma_{j}}^{K_{j}} g_{j} (\bar{x}),  d_{\theta})=0$, we get that \eqref{eqn:conecone} implies $0\leq \sigma(\partial_{\beta_{j}, \gamma_{j}}^{K_{j}} g_{j} (\bar{x}), \hat{v})$, $\forall \ \hat{v}$.  
        Now, by ($d.2$), one has $0 \notin \partial_{\beta_{j}, \gamma_{j}}^{K_{j}} g_{j} (\bar{x})$. So, the strong Hahn-Banach theorem implies the existence of $\bar{\alpha} \in \mathbb{R}$ and $\bar{d}$ (with $\|\bar{d}\|\leq 1$) such that $\langle w, \bar{d} \rangle \leq \bar{\alpha} <0$, $\forall w \in \partial_{\beta_{j}, \gamma_{j}}^{K_{j}} g_{j} (\bar{x})$. 
        Taking the supremum, one has $\sigma(\partial_{\beta_{j}, \gamma_{j}}^{K_{j}} g_{j} (\bar{x}), \bar{d})\leq \bar{\alpha}<0$ which is a contradiction.   
    \end{itemize}
 Therefore, in any case, there exists $r>0$ such that $\bar{x}+rd_{\theta} \in \Omega$. 
 
 Finally, from \eqref{eqn:hbsupremumcone}, we get $\langle v, (\bar{x} + r d_{\theta}) - \bar{x} \rangle = r \langle v, d_{\theta} \rangle>0 $ and thus $v \notin N (\Omega, \bar{x})$, contradicting the main supposition. 
\end{proof}

\begin{remark}\label{remark:normalcone}
Some remarks about the hypotheses of Theorem \ref{theo:normalcone} are in order:

\begin{enumerate}
    \item[(a)] A simple situation where \eqref{eqn:slatertypecone} holds is the following:
\begin{itemize}
  \item  Suppose that $\gamma_{j}>0$ and $K_{j}$ is convex for all $j$. 
       If there exists $\bar{y}\in \cap_{j\in I(\bar{x})} S_{g_{j}}(\bar{x})\cap \cap_{j\in I(\bar{x})} K_{j}$ with $\bar{y} \neq \bar{x}$, then \eqref{eqn:slatertypecone} holds.
  \end{itemize}
    In fact, we affirm that $\widehat{d}:=\bar{y}-\bar{x}$ satisfies \eqref{eqn:slatertypecone}.
    Due to the convexity of $K_{j}$, we get $\widehat{d} \in T(K_{j}, \bar{x})$ for all $j \in I(\bar{x})$. 
    Take any $w \in \partial_{\beta_{j}, \gamma_{j}}^{K_{j}} g_{j} (\bar{x})$. 
    Thus, it follows from \eqref{eqn:wj} that $\langle w, \widehat{d} \rangle \leq -(\beta_{j} \gamma_{j}/2) \|\widehat{ d}\|^{2} < 0$. 
    Hence, for every $j\in I(\bar{x})$, one has 
       \begin{align*}
    \sigma (\partial_{ \beta_{j}, \gamma_{j}}^{K_{j}} g_{j} (\bar{x}), \widehat{d}) \leq -(\beta_{j} \gamma_{j}/2) \|\widehat{ d}\|^{2} < 0.        
       \end{align*}

    \item [(b)] The inclusion given by Theorem \ref{theo:normalcone}(b) can be simplified under some assumptions. Indeed, set $K_{j} = S_{g_j} (\bar{x})$ for all $j \in I(\bar{x})$. By following the proof of \cite[Proposition 15]{Kab-Lara-2} under the proper modification, we see that if $\text{cone }(\partial_{\beta, \gamma_{j}}^{K_{j}} g_{j}(\bar{x}))$ is a closed set with $0 \notin \partial_{\beta, \gamma_{j}}^{K_{j}} g_{j}(\bar{x})$, then $S_{g_j}(\bar{x})$ is convex and $g_{j}$ is $F_{H }$-regular at $\bar{x}$ on $K_{j}$, thus $N(S_{g_j} (\bar{x}), \bar{x})\subset \text{cone }(\partial_{\beta, \gamma_{j}}^{K_{j}} g_{j}(\bar{x}))$ and  
        \begin{align}
     \left(\text{cone }\partial_{\beta, \gamma_{j}}^{K_{j}} g_{j} (\bar{x}) \right)^{\circ} \subset N(S_{g_j} (\bar{x}), \bar{x})^{\circ} = T(S_{g_j} (\bar{x}), \bar{x}).       
        \end{align}
    Then the inclusion given by $(b)$ is trivially satisfied. 
    
    \item [(c)] Note that if $\partial^{\infty} g_{j} (\bar{x})$ is convex, then $\partial^{\infty}g_{j} (\bar{x}) \subset N_{g_{j}} (\bar{x})$ and $K_{j} = S_{g_{j}} (\bar{x})$ for all $j \in I(\bar{x})$. As a consequence, the convex hull in \eqref{eqn:hullnormalconemaxequality} can be removed.
  \end{enumerate}
\end{remark}

We mention Theorem \ref{theo:normalcone} subsumes some of the results given in \cite[Lemma 28]{Kab-Lara-2}.
In fact, in the case, that $K_{j}:=\Omega$, $\forall j$; $\beta_{j}:=\beta$, $\gamma_{j}:=\gamma$, one has 
    \begin{align*}
        \partial_{\beta, \gamma}^{\Omega} g_{j}(\bar{x})=
          \{v \in \mathbb{R}^{n}: 
            \langle v, y-\bar{x}\rangle 
            \leq 
            -\frac{1}{2}\beta \gamma \|y-\bar{x}\|^{2}, \ \forall y \in \Omega 
          \}, \text{ for every } j \in I(\bar{x}). 
    \end{align*}   
 Further, we assume that $\partial^{\infty}g_{j}(\bar{x})=\emptyset$, $\forall j$ and $\Omega$ is not a singleton. 
 Thus under such assumptions, hypotheses (a) and (b) of Theorem \ref{theo:normalcone} are trivially satisfied. 
 We also note that the inclusion $\emptyset=\partial^{\infty} g_{j}(\bar{x}) \subset (S_{g_{j}}(\bar{x})-\bar{x})^{\circ}$ in hypotheses (c) of Theorem \ref{theo:normalcone} holds. 
 Finally, under the fulfillment of the hypotheses (c) and (d) of Theorem \ref{theo:normalcone}, one has the normal cone reduces to $N(\Omega, \bar{x})=\text{cone }  \partial_{\beta, \gamma}^{\Omega} g_{j}(\bar{x})$ for any $j\in I(\bar{x})$.   

 We end this section, by given an example where \eqref{eqn:hullnormalconemaxequality} holds where the convex sets $K_{j}$, $\forall j$ are different of $\Omega$. 
 In fact, we consider the feasible set $\Omega=\{x\in \mathbb{R}^{n}: g(x)\leq 0\}$ where $g:\mathbb{R}\rightarrow \mathbb{R}\cup\{\infty\}$ is defined in Example \ref{ex:itemii}. In this case, $\Omega=[-1, 0]$ and take $K=[-1, 1]$. 
 Here, $\partial_{1, 1}^{[-1, 1]}g(0)=[4^{-1}, 2]$ and $\partial_{1, 1}^{[-1, 0]}g(0)=[1/2, \infty)$.
 We note that   $N(\Omega, 0)=\mathbb{R}_{+}=\text{cone }\partial_{1, 1}^{[-1, 1]}g(0)=\text{cone }\partial_{1, 1}^{[-1, 0]}g(0)$. 

\section{Optimality Conditions}\label{sec:FJKKT}

We start this section with the next technical theorem relating the subdifferentials of the objective function with the normal cone of the feasible set at local minimizers. 

   \begin{theorem}\label{theo:fjnecessary}
Assume that $f$ is a lsc function and $\mathcal{C}$ is a closed convex set. Let $\bar{x} \in \mathcal{C}$ be a local solution of minimize $f$ over $\mathcal{C}$. 
Then, there exist $\gamma_{0}\geq 0$, $v \in \partial f(\overline{x})$ and $v^{\infty} \in \partial^{\infty} f(\overline{x})$ such that 
    \begin{align}\label{eqn:fj2}
 0 \in \gamma_{0} v + \widehat{\gamma}_{0} v^{\infty} + N(\mathcal{C}, \bar{x}),  
     \end{align}
where $\widehat{\gamma}_{0} = 1$ and $\| v^{\infty}\| = 1$ if $\gamma_{0}=0$; and  $\widehat{\gamma}_{0} = 0$ if $\gamma_{0} > 0$.
  \end{theorem}
   \begin{proof}
Let $\bar{x} \in \mathcal{C}$ be a local solution of minimize $f$ over $\mathcal{C}$. 
To derive \eqref{eqn:fj2}, we will use an argument based on the penalization method, an useful technique used to obtain FJ/KKT conditions and also sequential optimality conditions, see \cite{akktconic,borweinnecessary} and the references therein. Since $\bar{x}$ is a local minimizer, there exists $\delta > 0$ such that $\bar{x}$ is the unique global solution of
   \begin{align*}
\text{minimize } f(x) + \frac{1}{2}\|x-\bar{x}\|^{2} \text{ subject to } x \in \mathcal{C} \cap \mathbb{B} (\bar{x}; \frac{1}{2} \delta). 
   \end{align*}
Given $k\in \mathbb{N}$, consider the penalized optimization problem
    \begin{align}\label{eqn:penalized}
\text{minimize } f(x) + k \, \text{dist}(x, \mathcal{C})^{2}+\frac{1}{2}\|x-\bar{x}\|^{2}  \text{ subject to } x\in \mathbb{B}(\bar{x}; \frac{1}{2}\delta). 
   \end{align}

Let $y^{k}$ be a global solution of \eqref{eqn:penalized} (which exists as a consequence of the Weierstrass theorem). 
We will show that $y^{k}\rightarrow \bar{x}$. 
Let $\hat{x}$ be any limit point of $\{y^{k}\}$. 
By simplicity, assume that $y^{k}\rightarrow \hat{x}$. 
Clearly, $\|\hat{x}-\bar{x}\|\leq \frac{1}{2}\delta$.
Now, from \eqref{eqn:penalized}, one has 
   \begin{align}\label{eqn:penality2}
f(y^{k}) + k \, \text{dist}(y^{k}, \mathcal{C})^{2}+\frac{1}{2}\|y^{k}-\bar{x}\|^{2} 
       \leq 
f(\bar{x}) + k \, \text{dist}(\bar{x}, \mathcal{C})^{2}+\frac{1}{2}\|\bar{x}-\bar{x}\|^{2} 
       =f(\bar{x}) 
   \end{align}
The above implies $\, \text{dist}(y^{k}, \mathcal{C})^{2}\leq (f(\bar{x})-f(y^{k}))/k$.
By taking the limit we get $\text{dist}(\hat{x}, \mathcal{C})^{2}\leq 0$ and so $\hat{x}\in \mathcal{C}$. 
By \eqref{eqn:penality2}, one has $f(y^{k}) +\frac{1}{2}\|y^{k}-\bar{x}\|^{2} \leq f(\bar{x})$ and so $f(\hat{x}) +\frac{1}{2}\|\hat{x}-\bar{x}\|^{2} \leq f(\bar{x})+\frac{1}{2}\|\bar{x}-\bar{x}\|^{2}$ with 
$\hat{x}\in \mathbb{B}(\bar{x}, \frac{1}{2}\delta)$.
As $\bar{x}$ is the unique global minimizer, we must have $\bar{x}=\hat{x}$. The latter implies that $\{y^{k}\}$ has a  unique limit point and thus $y^{k}\rightarrow \bar{x}$. 

As $y^{k} \rightarrow \bar{x}$, one has $\|y^{k} - \bar{x}\| < \delta/2$ for $k$ large enough. 
Now, since the function $x \mapsto \text{dist} (x, \mathcal{C})^{2}$ is continuous differentiable, by using \cite[Exercise 10.10]{rwets}, the Fermat rule and since $\nabla [\text{dist} (\cdot, \mathcal{C})^{2}](y^{k})\in \text{dist} (y^{k}, \mathcal{C})\partial \text{dist} (\cdot, \mathcal{C})(y^{k})\subset N(\mathcal{C}, \text{proj}_{\mathcal{C}}(y^{k}))$ (due to the convexity of $\mathcal{C}$), we get 
   \begin{align}\label{eqn:fermat}
0 \in \widehat{\partial} \left( f(\cdot) + k \,  \text{dist} (\cdot, \mathcal{C})^{2} + \frac{1}{2} \|\cdot - \bar{x}\|^{2} \right) (y^{k}) 
& = \widehat{\partial} f(y^{k})+\nabla ( k \,  \text{dist} (\cdot, \mathcal{C})^{2} + \frac{1}{2} \|\cdot - \bar{x}\|^{2})(y^{k}) \nonumber \\
& \subset \widehat{\partial} f(y^{k})+\nabla ( k \,  \text{dist} (\cdot, \mathcal{C})^{2})(y^{k})+(y^{k} - \bar{x}) \nonumber \\
& \subset \widehat{\partial} f(y^{k}) + N(\mathcal{C}, \text{proj}_{\mathcal{C}}(y^{k})) + (y^{k} - \bar{x}).
   \end{align}

Clearly, $\text{proj}_{\mathcal{C}}(y^{k})\rightarrow \bar{x}$. 
Now, we set $\varepsilon_{k} := (y^{k} - \bar{x})$ for all $k \in \mathbb{N}$. Then, $\varepsilon_{k} \rightarrow 0$. 
From \eqref{eqn:fermat}, there exists a sequence $\{v^{k}\} \subset \mathbb{R}^{n}$ with $v^{k} \in \widehat{\partial} f(y^{k})$ such that $0 \in v^{k} + N(\mathcal{C}, \text{proj}_{\mathcal{C}}(y^{k})) + \varepsilon_{k}$. 

Now, we consider two cases depending on if $\{v^{k}\}$ is bounded or not. 

{\bf (i)}. Suppose that $\{v^{k}\}$ is bounded. In this case, after taking an adequate subsequence we suppose that $v^{k} \rightarrow v$ for some $v \in \mathbb{R}^{n}$. Clearly, $v \in \partial f(\bar{x})$. 
Hence, by taking the limit in $0\in v^{k}+N(\mathcal{C}, \text{proj}_{\mathcal{C}}(y^{k}))+ \varepsilon_{k}$, and by using the outercontinuity of the normal cone, we get $0 \in v+ N(\mathcal{C}, \bar{x})$.

{\bf (ii)}. Suppose that $\{v^{k}\}$ is not bounded. Then, after taking an adequate subsequence, we may suppose that $\|v^{k}\| \rightarrow \infty$ and $v^{k}/\|v^{k}\| \rightarrow v^{\infty}$ for some $v^{\infty} \in \partial^{\infty} f(\bar{x})$ with $\|v^{\infty}\|=1$. 
Taking the limit in $0\in (v^{k}/\|v^{k}\|) + N(\mathcal{C}, \text{proj}_{\mathcal{C}}(y^{k})) + (\varepsilon_{k}/\|v^{k}\|)$, and from the outercontinuity of the normal cone, we get $0 \in v^{\infty}+N(\mathcal{C}, \bar{x})$. 

Therefore, in both cases, we get $0 \in \gamma_{0}v + \widehat{\gamma}_{0} v^{\infty} + N(\mathcal{C}, \bar{x})$,  where $\widehat{\gamma}_{0} = 1$ and $\|v^{\infty}\| = 1$ (if $\gamma_{0} = 0$) and $\widehat{\gamma}_{0} = 0$ (if $\gamma_{0} > 0$).
\end{proof}

 \subsection{Applications to mathematical programming}
 \label{subsec:applicationmp}
 
In this section, we focus on deriving optimality conditions for mathematical programming problem of the form 
\begin{equation}\label{main:problem}
 \text{minimize} \ f(x) \ \text{subject to } g_{j} (x) \leq 0, \ j \in I, \tag{P}
\end{equation}
where $I$ is a finite index set. Here, we assume that the feasible set 
\begin{equation}\label{set:omega1}
 \Omega := \{x \in \mathbb{R}^{n} \mid \, g_{j} (x) \leq 0; ~ \forall ~ j \in I\},
\end{equation}
is a closed convex set and, as usual, we do not impose a local Lipschitz or continuity assumption on the constraint functions $g_{j}$ for all $j$.

 Let $\bar{x} \in \Omega$ be a feasible point, then the set of active indexes is given by $ I(\overline{x}) := \{j \in I: ~ g_{j} (\overline{x}) =0 \}$. Hence, we continue with the following notion of Fritz-John (FJ) and Karush-Kuhn-Tucker (KKT) conditions at a point $\bar{x}\in \Omega$.

\begin{definition}\label{def:fj}
 Let $\{K_{j}\}_{j \in I} \subset \mathbb{R}^{n}$ be a family of subsets, $\beta_{j}>0$ and $\gamma_{j} \geq 0$ for all $j \in I$. We say that a feasible point $\bar{x}$ satisfies the generalized Fritz-John (FJ) optimality condition if we can find scalars $\gamma_{0} \geq 0$ and $\mu_{j} \geq 0$, $j \in I(\bar{x})$ satisfying $\gamma_{0} + \sum_{j \in I(\bar{x})} \mu_{j}=1$ such that 
   \begin{align}\tag{G-FJ}\label{eqn:fj}
 0 \in \gamma_{0} \partial f(\overline{x}) + \widehat{\gamma}_{0} \partial^{\infty} f(\overline{x}) + \sum_{j \in I(\bar{x}): \mu_{j}>0} \mu_{j} \partial^{K_{j}}_{\beta_{j}, \gamma_{j}} g_{j} (\overline{x}) + \sum_{j \in I(\bar{x}): \mu_{j}=0} \partial^{\infty} g_{j} (\bar{x}), 
    \end{align}
where $\widehat{\gamma}_{0} = 1$ if $\gamma_{0} = 0$ and $\widehat{\gamma}_{0} = 0$ if $\gamma_{0} > 0$.

If $\gamma_{0}>0$, then we say that the generalized Karush-Kuhn-Tucker (KKT) conditions hold at $\bar{x}$. The scalars $\{\mu_{j}/\gamma_{0}\}_{j \in I(\bar{x})}$ are called {\it multipliers}.
\end{definition}

We point out that in our definition of FJ/KKT conditions, we use $\partial^{\infty} g(\bar{x})$ instead of $N_{g}(\bar{x})$ (note that if $g$ is strong quasiconvex with $\gamma>0$ on $\mathbb{R}^{n}$ and $S_{g}(\cdot)$ is inner semicontinuous at $\bar{x}$, then $\partial^{\infty} g(\bar{x}) \subset N_{g}(\bar{x})$ by Proposition \ref{propo:limitinghorizon}). 
Thus, if we use $N_{g_{j}}(\bar{x})$ instead of $\partial^{\infty} g_{j}(\bar{x})$ in Definition \ref{def:fj}, then we get a condition with more chances to be fulfilled at local minimizers, but, on the other hand, it can be too weak that non-minimizers may fulfill that condition. 
For instance, consider the next example. 


\begin{example}\label{ex:fj}
 Let $f, g: \mathbb{R} \rightarrow \mathbb{R} \cup \{+\infty\}$. We consider the problem of minimizing $f(x):=-x^{2}$ subject to $g(x) \leq 0$, where $g$ is defined by
\begin{align*}
 g(x) = \left\{
 \begin{array}{cl}
  0 & {\rm if} ~ x=0, \\
  - x^{-1} & {\rm if} ~ 0<x \leq 1, \\
  + \infty & {\rm otherwise}.
 \end{array}
 \right.     
\end{align*}
Here, $\Omega = [0, 1]$. Clearly, $\bar{x} = 0$ is not a local minimizer, $\nabla f(\bar{x}) = \partial f(\bar{x}) = 0$ and $\partial^{\infty} f(\bar{x}) = 0$. Note that $g$ is strongly quasiconvex with modulus $\gamma =1$ on $\mathbb{R}$ and $S_{g} (\cdot)$ is inner semicontinuous at $\bar{x}$. 
Furthermore, $N_{g}(0) = (S_{g} (\bar{x}) - \bar{x})^{\circ} = \mathbb{R}_{-}$, $\partial_{1, 1}^{\mathbb{R}} g(0) = (- \infty, - 1/2]$, $\widehat{\partial} g(0) = \emptyset$, $\partial g(0) = \emptyset$, $\partial^{\infty} g(0) = \emptyset$ and $\partial^{q}g(0)=\emptyset$.
Hence, $0 \in (\partial f(\bar{x}) + N_{g}(0)) = \mathbb{R}_{-}$ holds.
On the other hand,  $0 \in \partial f(\bar{x}) + \mu\partial_{1, 1}^{\mathbb{R}} g(0) = \mu (- \infty, -1/2]$ ($\mu > 0$) or $0 \in \partial f(\bar{x}) + \partial^{\infty} g(0)$ cannot hold at the non-minimizer $\bar{x} = 0$, thus, we can use \eqref{eqn:fj} to eliminate some non-minimizers.
\end{example}

We point out that there are many FJ/KKT conditions in the literature. A very general necessary condition (able to deal with non-Lipschitz functions) is the condition given by \cite{borweinnecessary}. The next example shows an instance when our definition can be applied while the result of \cite[Corollary 2.6]{borweinnecessary} can not. 

\begin{example}\label{ex:fj2}
 Let $f, g: \mathbb{R} \rightarrow \mathbb{R} \cup \{+\infty\}$. We consider the problem of minimizing $f(x) := x$ subject to $g(x) \leq 0$, where $g: \mathbb{R} \rightarrow \mathbb{R} \cup \{+\infty\}$ is defined in Example \ref{ex:fj}. 
 Clearly, $\bar{x} = 0$ is a global minimizer. We note that $\partial_{1, 1}^{\mathbb{R}} g(0) = (-\infty, -1/2]$ and $0 \in \partial f(\bar{x}) +\partial_{1, 1}^{\mathbb{R}} g(0)=(-\infty, 1/2]$ hold. 
 
 On the other hand, we note that \cite[Corollary 2.6]{borweinnecessary} says at a local minimizer one has that $0 \in \partial f(\bar{x}) + \mu \partial g(0)$ ($\mu>0$) or $0 \in \partial f(\bar{x}) + \partial^{\infty} g(0)$. Since $\partial^{\infty} g(0)$ and  $\partial g(0)$ are empty sets, we cannot apply the results of \cite[Corollary 2.6]{borweinnecessary}.
\end{example}

Before continuing, we consider the following {\it constraint qualification }(CQ), 
\begin{align}\label{eqn:CQ}\tag{GCQ}
 N(\Omega, \bar{x}) = \bigcup_{\mu \in \mathbb{R}_{+}^{ |I(\bar{x})|}} \left( \sum_{j \in I(\bar{x}): \, \mu_{j} >0} \mu_{j} \partial_{\beta_{j}, \gamma_{j}}^{K_{j}} g_{j}(\bar{x}) + \sum_{j \in I(\bar{x}): \, \mu_{j} = 0} \partial^{\infty} g_{j}(\bar{x}) \right). 
\end{align} 
We point out that \eqref{eqn:CQ} is valid under the hypotheses of Theorem \ref{theo:normalcone} and Remark \ref{remark:normalcone}(a).
The following theorem affirms that the generalized FJ conditions are a necessary optimality condition, under some assumptions. The proof follows directly from Theorem \ref{theo:fjnecessary} and \eqref{eqn:CQ}, thus it is omitted. 

\begin{theorem}\label{theo:fjnecessary1}
 Let $\bar{x} \in \Omega$ be a local solution of \eqref{main:problem} with $f$ lsc function and $\Omega$ being a closed convex set. If \eqref{eqn:CQ} holds at $\bar{x}$, then the FJ conditions hold at $\bar{x}$.
\end{theorem}

Now, let us focus on the conditions that ensure the validity of \eqref{eqn:CQ}. 

\begin{theorem}\label{theo:closure}
    Let $\bar{x} \in \Omega$ be a feasible point of \eqref{main:problem} such that 
     \begin{align}\label{eqn:normalconemaxequality}
N(\Omega, \bar{x})=
\overline{ 
             \bigcup_{\mu \in \mathbb{R}_{+}^{|I(\bar{x})|}} 
                 \left\{
                   \sum_{j\in I(\bar{x}): \mu_{j}>0} \mu_{j} \partial_{\beta_{j}, \gamma_{j}}^{K_{j}} g_{j}(\bar{x})
                    +
                    \sum_{j\in I(\bar{x}): \mu_{j}=0} \partial^{\infty}g_{j}(\bar{x}) 
                 \right\}
   }.   
          \end{align}  
    Assume that the following condition is satisfied:
      \begin{enumerate}
         \item [(a)] $N(\Omega, \bar{x})$ is pointed, and for every $j\in I(\bar{x})$, we have $0 \notin \partial_{\beta_{j}, \gamma_{j}}^{K_{j}} g_{j}(\bar{x})$, and $[\partial_{\beta_{j}, \gamma_{j}}^{K_{j}} g_{j}(\bar{x})]^{\infty}\cap (S_{g_{j}}(\bar{x})\cap K_{j}-\bar{x})^{\circ}=\{0\}$.
      \end{enumerate}
     Then \eqref{eqn:CQ} holds at $\bar{x}$.
\end{theorem}
  \begin{proof}
Take $w \in N(\bar{x}, \Omega)$. 
Thus, there exists a sequence $\{w^{k}\} \subset \mathbb{R}^{n}$ with $w^{k} \rightarrow w$ such that $w^{k} = \sum_{j \in I(\bar{x}):  \mu_{j}^{k}>0} \mu_{j}^{k}v^{k}_{j} + \sum_{j\in I(\bar{x}): \mu^{k}=0} v^{k,\infty}_{j}$, $v_{j}^{k} \in \partial_{\beta_{j}, \gamma_{j}}^{K_{j}} g_{j} (\bar{x})$ for all $j \in I(\bar{x})$ with $\mu_{j}^{k} > 0$ and $v_{j}^{k, \infty} \in \partial^{\infty} g_{j}(\bar{x})$ for all $j \in I(\bar{x})$ with $\mu_{j}^{k} =0$.

It remains to prove \eqref{eqn:CQ} under (a). 
Indeed, after taking an adequate subsequence, we assume $I^{+}:=\{j \in I(\bar{x}) \mid \, \mu_{j}^{k}>0\}$ and $I^{0} := \{j \in I(\bar{x}) \mid \, \mu_{j}^{k}=0\}$ for all $k$.
First, we will show that $\{\mu_{j}^{k} v_{j}^{k}\}_{j \in I^{+}}$ and $\{v_{j}^{k, \infty}\}_{j \in I^{0}}$ are bounded. 
By contradiction, we assume 
$$M_{k}:=\max\{\|\mu_{j}^{k}v_{j}^{k}\|, \|v_{i}^{k, \infty}\| \mid \, j \in I^{+}, i \in I^{0}\} \rightarrow \infty.$$ 

Then, $w^{k}/M_{k} = \sum_{j\in I^{+}} (\mu_{j}^{k} v^{k}_{j}/M_{k}) + \sum_{j\in I^{0}} (v^{k, \infty}_{j}/M_{k})$. 
We assume that $\mu_{j}^{k} v^{k}_{j}/M_{k} \rightarrow w_{j} \in N(\Omega, \bar{x})$ for all $j \in I^{+}$ (because \eqref{eqn:normalconemaxequality} holds) and $v^{k, \infty}_{j}/M_{k}\rightarrow w^{\infty}_{j} \in N(\Omega, \bar{x})$ for all $j \in I^{0}$. As $I^{+}\cup I^{0}$ is finite, the maximum of $M_{k}$ must be attained in $j_{0}\in I^{+}\cup I^{0}$, for infinite many $k$. 
So, after taking an adequate limit and since $0=\sum_{j\in I^{+}} w_{j}+\sum_{j\in I^{0}} w_{j}^{\infty}$, we get $-w_{j_{0}} \in N(\Omega, \bar{x})$ with $\|w_{j_{0}}\|=1$, or $-w^{\infty}_{j_{0}}\in N(\Omega, \bar{x})$ with $\| w^{\infty}_{j_{0}} \|=1$, which contradicts that $ N(\Omega, \bar{x})$ is pointed. As consequence, $\{\mu_{j}^{k} v_{j}^{k}\}_{j \in I^{+}}$ and $\{v_{j}^{k, \infty}\}_{j \in I^{0}}$ are bounded.

Now, we will show that $\{\mu_{j}^{k}\}_{j \in I^{+}}$ is bounded. Indeed, as $\{\mu_{j}^{k}v_{j}^{k}\}_{j \in I^{+}}$ is bounded, we get $\|v_{j}^{k}\|\leq \frac{1}{\mu_{j}^{k}} M$ for some $M>0$. Thus, if $\mu_{j}^{k} \rightarrow \infty$ for some $j \in I^{+}$ we get $\|v_{j}^{k}\| \leq \frac{1}{\mu_{j}^{k}} M\rightarrow 0$ and $0 \in \partial_{\beta_{j}, \gamma_{j}}^{K_{j}} g_{j}(\bar{x})$, which is a contradiction. Therefore, $\{\mu_{j}^{k}\}_{j \in I^{+}}$ is bounded. Then, after taking an adequate subsequence, we assume $\mu_{j}^{k}\rightarrow \hat{\mu}_{j}$ for all $j \in I^{+}$.

Set $I^{++} := \{j \in I^{+} \mid \, \hat{\mu}_{j} > 0\}$ and $I^{+0} := \{j \in I^{+} \mid \, \hat{\mu}_{j} = 0\}$. 
Clearly, as $\{\mu_{j}^{k}v_{j}^{k}\}_{j \in I^{+}}$ is bounded, we get $\{v_{j}^{k}\}_{j \in I^{++}}$ is bounded, too. 

Let us prove that $\{v_{j}^{k}\}_{j \in I^{+0}}$ is bounded. Indeed, suppose that $\|v_{j}^{k}\| \rightarrow \infty$ for some $j \in I^{+0}$. Since $v_{j}^{k} \in \partial_{\beta_{j}, \gamma_{j}}^{K_{j}} g_{j}(\bar{x})$, we have 
\begin{align*} 
 \frac{\lambda}{\beta} \left \langle \frac{v_{j}^{k}}{\|v_{j}^{k}\|}, y - \bar{x} \right \rangle \leq - \frac{\lambda(1-\lambda)}{2} \frac{\gamma_{j}\|y-\bar{x}\|^{2}}{\|v_{j}^{k}\|} + \frac{\lambda^{2}}{2 \beta} \frac{\mu_{j}^{k} \|\bar{y} - \bar{x} \|^{2}}{\|v_{j}^{k}\|}, \ \forall ~ y \in S_{g_{j}}(\bar{x}) \cap K_{j}. 
\end{align*}

Suppose that $(v_{j}^{k}/\|v_{j}^{k}\|) \rightarrow w_{j}$ for some $w_{j} \in \mathbb{R}^{n}$, $\|w_{j}\| = 1$. Taking the limit in the above expression, we obtain
$$w \in (S_{g_{j}} (\bar{x}) \cap K_{j} - \bar{x})^{\circ},$$ 
and thus $w \in [\partial_{\beta_{j}, \gamma_{j}}^{K_{j}} g_{j}(\bar{x})]^{\infty} \cap(S_{g_{j}}(\bar{x}) \cap K_{j} - \bar{x})^{\circ}$ with $w \neq 0$, a contradiction. Therefore, $\{\mu_{j}^{k}\}_{j \in I^{+}}$, $\{v_{j}^{k}\}_{j \in I^{+}}$ and $\{v_{j}^{k, \infty}\}_{j \in I^{0}}$ are bounded. 

Finally, the result follows from $w^{k} = \sum_{j \in I^{++}} \mu_{j}^{k} v^{k}_{j} + \sum_{j \in I^{+0}} \mu_{j}^{k} v^{k}_{j} + \sum_{j \in I^{0}} v^{k, \infty}_{j}$, after taking the limit of an adequate subsequence.
\end{proof}

We point out that it is possible to obtain a description of $N(\Omega, \bar{x})$ in terms of strong subdifferentials without using the horizon subdifferential $\partial^{\infty} g_{j}(\bar{x})$ for all $j$. In fact, if $N(\Omega, \bar{x}) = \overline{ \bigcup_{\mu \in \mathbb{R}_{+}^{|I(\bar{x})|}} \left\{\sum_{j \in I(\bar{x}): \mu_{j} > 0} \mu_{j} \partial_{\beta_{j}, \gamma_{j}}^{K_{j}} g_{j} (\bar{x}) \right\}}$, then, under Assumption (a) of Theorem \ref{theo:closure} and following the argument of the proof of Theorem \ref{theo:closure},  with the proper modification, it is possible to show that 
  \begin{align*}
 N(\Omega, \bar{x}) = \{0\} \cup \bigcup_{\mu \in \mathbb{R}_{+}^{ |I(\bar{x})|}} \left\{ \sum_{j \in I(\bar{x}): \mu_{j}>0} \mu_{j} \partial_{\beta_{j}, \gamma_{j}}^{K_{j}} g_{j} (\bar{x}) \right\}.      
  \end{align*}
Indeed, as a direct application, we consider the next example. 

\begin{example}\label{ex:itemii}
Let us consider the function $g: \mathbb{R} \rightarrow \mathbb{R}$ defined by 
 \begin{align*}
  g(x) = \left\{
  \begin{array}{rl}
  2x & {\rm if} ~ x > 0, \\
  0 & {\rm if} ~ x=0,  \\
  -x-1 & {\rm if} ~ x<0.
 \end{array}
 \right.     
\end{align*}

Take $\Omega := \{x \in \mathbb{R} \mid g(\bar{x}) \leq 0\} = [-1, 0]$. Clearly, $N(\Omega, \bar{x}) = \mathbb{R}_{+}$, $N_{g} (0) = \mathbb{R}_{+}$, $\partial_{1, 1}^{[-1, 1]} g(0) = [4^{-1}, 2]$ and $N(\Omega, 0) = \overline{\bigcup_{\mu \in \mathbb{R}_{++}} \mu [4^{-1}, 2]}$. 
Here, item (a) of Theorem \ref{theo:closure} holds at $\bar{x}=0$, since $N(\Omega, \bar{x})$ is pointed, $0 \notin \partial_{1, 1 }^{[-1, 1]} g(0)$ and $[\partial_{1, 1 }^{[-1, 1]} g(0)]^{\infty} \cap (S_{g} (\bar{x})-\bar{x})^{\circ} =([4^{-1}, 2])^{\infty} \cap N_{g}(0) = \{0\} \cap \mathbb{R}_{+} = \{0\}$. 
Thus, we have $N(\Omega, 0) = \{0\} \cup \bigcup_{\mu \in \mathbb{R}_{++}} \mu [4^{-1}, 2]$. 
\end{example}

Now, consider the following assumption to ensure the fulfillment of the KKT conditions. Let $\bar{x}\in \Omega$ be a feasible point. Then we say that $\partial^{\infty}$-no nonzero abnormal multiplier condition holds, if 
\begin{align}\label{eqn:cq1}\tag{$\partial^{\infty}$-NNAMC}
\begin{array}{ll}
& 0=w+\sum_{j\in I^{+}} \mu_{j}v_{j} + \sum_{j\in I^{0}} v_{j}^{\infty}  \\
    & w \in \partial^{\infty}f(\bar{x}) \\ 
    & v_{j}\in \partial_{\beta_{j}, \gamma_{j}}^{K_{j}} g_{j}(\bar{x}), \ j \in I^{+}:=\{j \in I(\bar{x}):\mu_{j}>0\} \\
    & v_{j}^{\infty} \in \partial^{\infty}g_{j}(\bar{x}), \ j \in I^{0}:=\{j \in I(\bar{x}):\mu_{j}=0\}
      \end{array}
   \implies 
      \begin{array}{ll}
           & w=0,  \\
           & I^{+}=\emptyset ~~ \text{ and } \\
           & v_{j}^{\infty}=0, \ \forall \, j \in I^{0}.
      \end{array}
 \end{align}

From Theorem \ref{theo:fjnecessary}, we get the next result. 
\begin{theorem}\label{theo:kkt}
    Let $\bar{x} $ be a local minimizer of \eqref{main:problem}. 
    If \eqref{eqn:CQ} and \eqref{eqn:cq1} conditions hold, then the generalized KKT conditions is valid at $\bar{x}$.
\end{theorem}

Now, we turn our attention to deriving sufficient conditions based on the generalized FJ/KKT conditions. We start with the following theorem. 

\begin{theorem}\label{theo:strongsuffstrong}(Sufficient condition)
Let $\bar{x} \in \Omega$ be a feasible point satisfying the generalized FJ, that is, there exist scalars $\gamma_{0}$, $\mu_{j}\geq 0$ for all $j \in I(\bar{x})$, and vectors $v\in \partial f(\overline{x})$, $v^{\infty} \in \partial^{\infty} f(\overline{x})$, $v_{j}\in \partial^{K_{j}}_{\beta_{j}, \gamma_{j}} g_{j}(\overline{x})$ (if $\mu_{j}>0$) and $v^{\infty}_{j}\in \partial^{\infty}g_{j}(\bar{x})$ (if $\mu_{j}=0$) such that $\gamma_{0} +\sum_{j \in I(\bar{x})}\mu_{j}=1$ and 
\begin{align*}
0=\gamma_{0}v+\widehat{\gamma}_{0} v^{\infty}+
    \sum_{j \in I(\bar{x}): \mu_{j}>0} \mu_{j}v_{j}+
    \sum_{j \in I(\bar{x}): \mu_{j}=0} v_{j}^{\infty}, 
\end{align*}
where $\widehat{\gamma}_{0} = 1$ if $\gamma_{0} = 0$ and $\widehat{\gamma}_{0} = 0$ if $\gamma_{0} > 0$.

Suppose that $\Omega \subset \cap_{j \in I(\bar{x})}K_{j}$, that $\{g_{j}\}_{j \in I(\bar{x})}$ are strongly quasiconvex functions with modulus $\gamma_{j }\geq 0$ for all $j \in I(\bar{x})$, and $S_{g_{j}} (\cdot)$ is inner semicontinuous at $\bar{x}$ for all $j \in I(\bar{x})$.  
Then, there exists an open convex neighborhood $V$ of $\bar{x}$ such that   
        \begin{align}\label{eqn:strongkktsuff0}
    \bar{\mu}\|y-\bar{x}\|^{2} 
    \leq 
    \gamma_{0} \langle v, y-\bar{x} \rangle+ \widehat{\gamma}_{0} \langle v^{\infty}, y-\bar{x} \rangle, ~ \forall ~ y \in \Omega\cap V, 
        \end{align}  
where $\bar{\mu}\!:=\!\frac{1}{2} \sum_{j \in I(\bar{x}): \mu_{j}>0} \beta_{j} \gamma_{j} \mu_{j}$. 
Further, if $\bar{x}$ is a generalized KKT point then 
\begin{align}\label{eqn:strongkktsuff01}
\frac{1}{\gamma_{0}}\bar{\mu}\|y-\bar{x}\|^{2} \leq \langle v, 
y-\bar{x}\rangle, ~ \forall ~ 
y \in \Omega\cap V.
\end{align} 
\end{theorem}
   \begin{proof}
First, set $I^{+} := \{j \in I(\bar{x})\mid \, \mu_{j}>0\}$ and $I^{0} := \{j \in I(\bar{x})\mid \, \mu_{j}=0\}$. By Proposition \ref{propo:limitinghorizon}, for every $j \in I^{0}$ we have $v_{j} \in (S_{g_{j}}(\bar{x}) \cap V_{j}-\bar{x})^{\circ}$ for some $V_{j}$ open convex neighborhood of $\bar{x}$. Thus, $\langle v_{j}^{\infty}, y-\bar{x} \rangle \leq 0$ for all $y \in \cap_{j \in I^{0}}S_{g_{j}} (\bar{x}) \cap V_{j}$. 

On the other hand, for $j \in I^{+}$, since $v_{j}\in \partial^{K_{j}}_{\beta_{j}, \gamma_{j}} g_{j}(\overline{x})$, we obtain $\langle v_{j}, y-\bar{x}\rangle\leq -(\beta_{j}\gamma_{j}/2)\|y-\bar{x}\|^{2}$ for every $y \in \cap_{j \in I^{+}}S_{g_{j}}(\bar{x}) \cap K_{j}$.
By adding the expressions, we get 
   \begin{align}\label{eqn:inequa}
 \sum_{j\in I^{+}} \mu_{j} \langle v_{j}, y-\bar{x} \rangle
 +\sum_{j \in I^{0}} \langle v_{j}^{\infty}, y-\bar{x} \rangle 
 \leq -\frac{1}{2}\sum_{j\in I^{+}} \mu_{j}\beta_{j}\gamma_{j}\|y-\bar{x}\|^{2},
   \end{align}
for every $y \in (\cap_{j \in I^{+}}S_{g_{j}}(\bar{x})\cap K_{j})\cap(\cap_{j\in I^{0}}S_{g_{j}}(\bar{x})\cap V_{j})=\Omega \cap V$ where $V:=\cap_{j \in I^{0}}V_{j}$ is an open convex neighborhood of $\bar{x}$. 
From \eqref{eqn:inequa}, we get  
\begin{align*}
 \frac{1}{2}\sum_{j \in I^{+}} \mu_{j}\beta_{j}\gamma_{j}\|y-\bar{x}\|^{2} 
 & \leq - \sum_{j \in I^{+}} \mu_{j} \langle v_{j}, y- \bar{x} \rangle - \sum_{j \in I^{0}} \langle v_{j}^{\infty}, y- \bar{x} \rangle \\ \nonumber
 & \leq \left \langle - \sum_{j \in I^{+}} \mu_{j} v_{j}- \sum_{j \in I^{0}}  v_{j}^{\infty}, y-\bar{x} \right \rangle \\ \nonumber
 & = \gamma_{0} \langle v, y-\bar{x} \rangle + \widehat{\gamma}_{0} \langle v^{\infty}, y-\bar{x} \rangle, ~ \forall ~ y \in \Omega \cap V.
\end{align*}
The above implies \eqref{eqn:strongkktsuff0}. 
Note that \eqref{eqn:strongkktsuff01} follows from \eqref{eqn:strongkktsuff0} when $\widehat{\gamma}_{0}=0$. 
\end{proof}

We point out that Theorem \ref{theo:strongsuffstrong} gives a sufficient optimality condition under classical assumptions. In fact, consider the following statement. 

\begin{theorem}\label{theo:suffclassical}
Let $f$ be a locally Lipchitz and $\alpha$-strongly pseudoconvex function with respect to $\partial f(\bar{x})$ (that is, $\text{sup}_{\xi \in \partial f(\bar{x})}\langle \xi, y-\bar{x} \rangle\geq 0$ implies $f(y)\geq f(\bar{x})+\alpha\|y-\bar{x}\|^{2}$). 
 
 Then, if $\bar{x}$ is a feasible point that satisfies the hypothesis of Theorem \ref{theo:strongsuffstrong} with $f$ locally Lipchitz and $\alpha$-strongly pseudoconvex, then $f(y)\geq f(\bar{x})+\alpha\|y-\bar{x}\|^{2}$ for all $y \in \Omega\cap V$, for some open neighborhood $V$ of $\bar{x}$. 
\end{theorem}
    \begin{proof}
 In fact, since $f$ is locally Lipschitz we have $\partial^{\infty}f(\bar{x})=\emptyset$.
 Now, from expression \eqref{eqn:strongkktsuff0}, we get $0\leq \bar{\mu}\|y-\bar{x}\|^{2}\leq \gamma_{0} \text{ sup}_{v \in \partial f(\bar{x})}\langle v, y-\bar{x}\rangle$, for every $y \in V\cap \Omega$.
 As $\gamma_{0}>0$ and due to the $\alpha$-strongly pseudoconvexity of $f$, we get that $f(y)\geq f(\bar{x})+\alpha\|y-\bar{x}\|^{2}$ for all $y \in \Omega\cap V$.       
    \end{proof}
    
\begin{remark}\label{remark:1}      
We mention that convex functions are $\alpha$-strongly pseudoconvex for $\alpha=0$ and $\alpha$-strongly quasiconvex functions are $\alpha$-strongly pseudoconvex, too, (see \cite{ADSZ,KS} for more on strongly pseudoconvex functions). 
\end{remark}

The following simple example shows an instance in which our results can be applied while no other KKT result does.

\begin{example}\label{ex:sufficiency}
 Let $f, g: \mathbb{R} \rightarrow \mathbb{R} \cup \{+\infty\}$. We consider the problem of minimizing $f(x) = x$ subject to $g(x)\leq 0$, where $g$ is given as in Example \ref{ex:fj}. Observe that 
    \begin{align*}
 0 \in \nabla f(0) + \partial_{1, 1}^{\mathbb{R}} g(0) = 1 + (-\infty, -1/2]=(-\infty, 1/2].       
    \end{align*}
 
 Since $f$ is convex, it follows from Remark \ref{remark:1} that $\bar{x} = 0$ is an optimality solution. Here, $\partial^{\infty} g(0) = \partial g(0) = \widehat{\partial} g(0) = \partial^{q}g(0)=\emptyset$. 
 Here, Theorem \ref{theo:strongsuffstrong} can be applied even in situations where classical sufficient KKT conditions fail. 
\end{example}

Another interesting application of our sufficient condition is the following result, which concerns quadratic fractional programming. 

\begin{proposition}\label{propo:sufficiencyquadratic}
Consider the quadratic fractional programming \eqref{eqn:quadratic} optimization 
\begin{align}\label{eqn:quadratic}\tag{QFP}
 {\rm minimize } \ f(x) ~ {\rm subject ~ to } ~ \frac{g_{1}(x)}{g_{2}(x)} \leq \alpha,      
\end{align}
where $\alpha \in \mathbb{R}$, $g_{1}, g_{2}: \mathbb{R}^{n} \rightarrow \mathbb{R}$    are functions such that $g_{1}$ is strongly convex with modulus $\gamma>0$, and $g_{2}$ is a positive affine function. 
Let $m, M$ be two positive numbers with $m\leq M$. We suppose that $K := \{x \in \mathbb{R}^{n}: m\leq g_{2}(x) \leq M\} \neq \emptyset$.

If there exists a feasible point $\bar{x}$ of \eqref{eqn:quadratic}, vectors $v \in \partial f(\bar{x})$, $v^{\infty} \in \partial^{\infty} f(\bar{x})$ and scalars $\gamma_{0}, \hat{\gamma}_{0}$, $\mu>0$ satisfying $g_{1}(\bar{x}) = \alpha g_{2} (\bar{x})$, $\gamma_{0} + \mu=1$ such that 
\begin{align}\label{eqn:suffquadratic}
 0 \in \gamma_{0} v + \hat{\gamma}_{0} v^{\infty} + \mu \partial^{FM} (g_{1} - \alpha g_{2}) (\bar{x}), 
\end{align}
where $\widehat{\gamma}_{0} = 1$ if $\gamma_{0} = 0$ and $\widehat{\gamma}_{0} = 0$ if $\gamma_{0} > 0$.
       
Then, we can find an open convex neighborhood $V$ of $\bar{x}$ such that
\begin{align*}
\frac{1}{2}\gamma \mu \|y-\bar{x}\|^{2} \leq  \gamma_{0}\langle v, y - \bar{x} \rangle +
\gamma_{0} \langle v, y - \bar{x} \rangle, \ \ \forall y \in \Omega \cap V.
\end{align*}
In particular, if $f$ is a locally Lipchitz and $\alpha$-strongly pseudoconvex function with respect to $\partial f(\bar{x})$, then $f(y)\geq f(\bar{x})+\alpha\|y-\bar{x}\|^{2}$ for all $y \in \Omega\cap V$. 
\end{proposition}
   \begin{proof}
The proof is a consequence of Theorem \ref{theo:strongsuffstrong}. In fact, first, we observe that $g=g_{1}-\alpha g_{2}$ is strongly convex of modulus $\gamma>0$. 
Thus, by \cite[Proposition 4.1]{LMT}, if we set $h(x)=g(x)/g_{2}(x)$, we have $M^{-1}\rho\  \partial^{FM}(g-h(\bar{x})g_{2}) (\bar{x})\subset \partial_{\rho, M^{-1}\gamma}^{K} h(\bar{x})$, $\forall \rho>0$. Hence, we get 
          \begin{align*}
0 \in \gamma_{0} v + \hat{\gamma}_{0} v^{\infty} + \mu \partial^{FM} (g_{1} - \alpha g_{2}) (\bar{x}) 
\subset 
\gamma_{0} v + \hat{\gamma}_{0} v^{\infty} + 
              \frac{\mu M}{\rho} \partial_{\rho, M^{-1}\gamma}^{K} h(\bar{x}).
          \end{align*}
 From \eqref{eqn:strongkktsuff0}, one has $\bar{\mu} \|y-\bar{x}\|^{2} \leq  \gamma_{0}\langle v, y - \bar{x} \rangle +\gamma_{0} \langle v, y - \bar{x} \rangle$ where $\bar{\mu}=\frac{1}{2}\frac{\mu M}{\rho}.\rho.\frac{\gamma}{M}=\frac{1}{2}\mu \gamma$.
   \end{proof}



\section{Conclusions}\label{sec:conclusion}
Our KKT and FJ type optimality conditions confirm that the strong subdifferential is a good tool when we deal with strongly quasiconvex functions, as noted in \cite{Kab-Lara-2} (see also \cite{ChLM,LMT}).
The proposed optimality conditions are established under mild assumptions and do not rely on classical convexity, which broadens their applicability. In particular, the FJ conditions allow us to handle situations where standard constraint qualifications may fail, while the KKT conditions offer sharper characterizations when suitable regularity assumptions are satisfied.

These findings contribute to a deeper understanding of the structure of strongly quasiconvex problems and open the door to further developments. Possible directions for future research include the study of algorithmic consequences of the proposed conditions, extensions to vector and multiobjective optimization, and applications to nonsmooth problems arising in applied mathematics and engineering.
\backmatter

\section{Declarations}



\subsection{Availability of supporting data}
No data sets were generated during the current study. 

\subsection{Author Contributions}
Both authors contributed equally to the study conception, design, wrote and corrected the manuscript.

\subsection{Competing Interests}
There are no conflicts of interest or competing interests related to this manuscript.

\subsection{Funding}
This research was partially supported by UTA research project Fortalecimiento de Gru\-pos de Investigaci\'on C\'odigo 8802-25, by ANID-Chile under Fondecyt Regular 1241040 (Lara) and by Fondecyt Regular 1231188 (Ramos). \newline

{\bf Acknowledgements.} The authors wishes to thank the reviewers for the comments and remarks that improved
the quality of this paper.


\begin{thebibliography}{9}       


\bibitem{akktconic}
 \textsc{R. Andreani, W. Gómez, G. Haeser, L. M. Mito, A. Ramos}, On optimality conditions for nonlinear conic programming, {\it Math. Oper. Res.}, \textbf{47 (3)}, 2160--2185, (2022).





\bibitem{AE}
 \textsc{K. J. Arrow, A. C. Enthoven}, Quasiconcave programming, {\it Econo\-me\-tri\-ca}, \textbf{29}, 779--800, (1961).

\bibitem{AusselSVVA2015}
 \textsc{D. Aussel, M. Pistek}, Limiting normal operator in quasiconvex analysis, {\it Set-Valued Var. Anal.}, \textbf{23}, 669--685, (2015).
 
\bibitem{ADSZ}
 {\sc M. Avriel, W. E. Diewert, S. Schaible, I. Zang}. ``Generalized Concavity''. SIAM, Philadelphia, (2010).

 
\bibitem{borweinnecessary}
 {\sc J. M. Borwein, J. S. Treiman, Q. J. Zhu}, Necessary conditions for constrained optimization problems with semicontinuous and continuous data, {\it Trans. Amer. Math. Soc.}, {\bf 350 (6)}, 2409--2429, (1998).
 

\bibitem{CM-Book}
 {\sc A. Cambini, L. Martein}. ``Generalized Convexity and Optimization: Theo\-ry and Applications''. Springer, (2009).

\bibitem{ChLM}
 {\sc J. Choque, F. Lara, R. T. Marcavillaca}, A subgradient projection method for quasiconvex minimization. {\it Positivity}, Vol. 28, Issue 5, Paper 64, (2024). 

\bibitem{appropriatesubqc}
 {\sc A. Daniilidis, N. Hadjisavvas, J. E. Mart\'inez-Legaz}, An appropriate subdifferential for quasiconvex functions, {\it SIAM J. Optim.}, {\bf 12}, 407--420, (2001).

 \bibitem{D-1959}
 \textsc{G. Debreu}. ``Theory of value". John Wiley, New York,
 (1959).



\bibitem{FS}
{\sc J. Frenk, S. Schaible}, Fractional programming. In: N. Hadjisavvas et al. (eds.): ``Handbook of Generalized Convexity and Generalized Monotonicity". pp. 335--386. Springer-Verlag, Boston, (2005).

\bibitem{gejiangyecomplexitylp}
 {\sc D. Ge and X. Jiang and Y. Ye}, A note on complexity of Lp minimization. {\it Math. Programm.}, {\bf 129}, 285--299, (2011).
 
\bibitem{HKS}
 {\sc N. Hadjisavvas, S. Komlosi, S. Schaible}. ``Handbook of Generalized Convexity and Generalized Monotonicity''. Springer-Verlag,
 Boston, (2005).

\bibitem{IL-7}
 {\sc A. Iusem, F. Lara}, Proximal point algorithms for quasiconvex pseudomonotone equilibrium problems. {\it J. Optim. Theory Appl.}, {\bf 193}, 443--461, (2022). 

 \bibitem{ILMY}
 {\sc A. Iusem, F. Lara, R. T. Marcavillaca, L.H. Yen}, A two-step PPA for nonconvex equilibrium problems with applications to fractional programming, {\it J. Global Optim.}, {\bf 90}, 755--779, (2024).

\bibitem{Kab-Lara-2}
 {\sc A. Kabgani, F. Lara}, Strong subdifferentials: theory and applications in nonconvex optimization, {\it J. Global Optim.}, {\bf 84}, 349--368, (2022).

\bibitem{KS} 
 {\sc S. Karamardian, S. Schaible}, Seven kinds of monotone maps, {\it J. Optim. Theory Appl.}, {\bf 66}, 37--46, (1990).  

 
\bibitem{Lara-9}
 {\sc F. Lara}, On strongly quasiconvex functions: existence results and proximal point algorithms, {\it J. Optim. Theory Appl.}, {\bf 192}, 891--911, (2022).

\bibitem{LMT}
 {\sc F. Lara, R. T. Marcavillaca, T. V. Thang}, A subgradient projection method for quasiconvex multiobjetive optimization, {\it Optimization, DOI:  10.1080/02331934.2024.2436577}, (2025). 

\bibitem{LMV}
 {\sc F. Lara, R. T. Marcavillaca, P. T. Vuong}, Characterization, dynamical systems and gradient methods for strongly quasiconvex functions, {\it J. Optim. Theory Appl.}, {\bf 206}, article number 60, (2025).  

\bibitem{LMY}
 {\sc F. Lara, R. T. Marcavillaca, L. H. Yen}, An extragradient projection method for strongly quasiconvex equilibrium problems with applications. {\it Comp. Appl. Math.}, {\bf 43}, Issue 3, N.º 128, 21 pp. (2024). 


\bibitem{Manga}
 {\sc O. L. Mangasarian}. ``Nonlinear Programming''. SIAM, Classics in 
 Applied Mathematics, Philadelphia, (1994).

\bibitem{MWG}
 {\sc A. Mas-Colell, M.D. Whinston, J.R. Green}, {\it Microeconomic Theory}. Oxford University Press, Oxford, (1995). 


 \bibitem{NS} 
  {\sc N. M. Nam, J. Sharkasky}, On strong quasiconvexity of functions in infinite dimensions, {\it Optim. Letters, DOI: 10.1007/s11590-025-02261-x}, (2025). 

\bibitem{Pi} 
{\sc N. Pischke}, On the proximal point algorithm for strongly quasiconvex functions in Hadamard spaces, {\it Optim. Methods $\&$ Software, DOI: 10.1080/10556788.2025.2531481}, (2025). 

\bibitem{P}
 \textsc{B. T. Polyak}, Existence theorems and convergence of minimizing
 sequences in extremum problems with restrictions, \textit{Soviet Math.},
 \textbf{7}, 72--75, (1966). 

 \bibitem{rwets}
 \textsc{R. T. Rockafellar, R. Wets}, Variational Analysis, \textit{Series: Grundlehren der mathematischen Wissenschaften}, \textbf{317}, German, (2009).


\bibitem{suzuki}
  {\sc S. Suzuki}. 
  ``Karush-Kuhn-Tucker type optimality condition for quasiconvex programming in terms of the Greenberg-Pierskalla subdifferential". 
  \textit{J. Global Optim.}, \textbf{79}, 191–-202, (2021).

\bibitem{Tran}
 {\sc V. N. Tran}, A relaxed proximal point algorithm with double-inertial effects for nonconvex equilibrium problems, arXiv preprint, arXiv: 2502.10986, (2025).

\bibitem{yaoeuro}
  {\sc P. Q. Khanh, H. T. Quyen, J-C. Yao}. 
  ``Optimality conditions under relaxed quasiconvexity assumptions using star and adjusted subdifferentials", \textit{European J. Oper. Res.}, \textbf{212}, 235--241, (2011).
\end{thebibliography}

\end{document}